\documentclass[letterpaper,10pt]{article}
\usepackage{fullpage}

\usepackage{graphicx}%
\usepackage{multirow}%
\usepackage{amsmath,amssymb,amsfonts}%
\usepackage{amsthm}%
\usepackage{mathrsfs}%
\usepackage[title]{appendix}%
\usepackage{xcolor}%
\usepackage{textcomp}%
\usepackage{manyfoot}%
\usepackage{booktabs}%
\usepackage{listings}%
\usepackage{caption}
\usepackage{microtype}

\usepackage[numbers,sort&compress]{natbib}

\usepackage{hyperref}
\usepackage{url}
\usepackage{mathtools}

\usepackage{comment}
\usepackage{siunitx}
\usepackage{relsize}
\usepackage{ifthen}
\usepackage[colorinlistoftodos]{todonotes}

\usepackage[caption=false]{subfig}

\usepackage[vlined,ruled,linesnumbered]{algorithm2e}
\usepackage{graphics} 
\usepackage{rotating}
\usepackage{color}
\usepackage{enumerate}
\usepackage[T1]{fontenc}
\usepackage{psfrag}
\usepackage{epsfig} 
\usepackage{booktabs}
\usepackage{graphicx,url}
\usepackage{multirow}
\usepackage{array}
\usepackage{latexsym}
\usepackage{amsfonts}
\usepackage{amsmath}
\usepackage{amssymb}
\usepackage{mathtools}
\usepackage{xstring}
\usepackage{multirow}
\usepackage{xcolor}
\usepackage{prettyref}
\usepackage{flexisym}
\usepackage{bigdelim}
\usepackage{breqn} 
\usepackage{listings}

\usepackage{enumitem}
\usepackage{xspace}
\usepackage{bm}
\graphicspath{{./figures/}}
\usepackage{tikz}
\usetikzlibrary{matrix,calc}

\usepackage{amsmath,array,arydshln,xparse}
\usepackage{adjustbox}

\usepackage{mdwlist}

\makecompactlist{itemize}{stditemize}

\newrefformat{prob}{Problem\,\ref{#1}}
\newrefformat{def}{Definition\,\ref{#1}}
\newrefformat{sec}{Section\,\ref{#1}}
\newrefformat{sub}{Section\,\ref{#1}}
\newrefformat{prop}{Proposition\,\ref{#1}}
\newrefformat{app}{Appendix\,\ref{#1}}
\newrefformat{alg}{Algorithm\,\ref{#1}}
\newrefformat{cor}{Corollary\,\ref{#1}}
\newrefformat{thm}{Theorem\,\ref{#1}}
\newrefformat{lem}{Lemma\,\ref{#1}}
\newrefformat{fig}{Fig.\,\ref{#1}}
\newrefformat{tab}{Table\,\ref{#1}}

\newcommand{\bdmath}{\begin{dmath}}
\newcommand{\edmath}{\end{dmath}}
\newcommand{\beq}{\begin{equation}}
\newcommand{\eeq}{\end{equation}}
\newcommand{\bdm}{\begin{displaymath}}
\newcommand{\edm}{\end{displaymath}}
\newcommand{\bea}{\begin{eqnarray}}
\newcommand{\eea}{\end{eqnarray}}
\newcommand{\beal}{\beq \begin{array}{ll}}
\newcommand{\eeal}{\end{array} \eeq}
\newcommand{\beas}{\begin{eqnarray*}}
\newcommand{\eeas}{\end{eqnarray*}}
\newcommand{\ba}{\begin{array}}
\newcommand{\ea}{\end{array}}
\newcommand{\bit}{\begin{itemize}}
\newcommand{\eit}{\end{itemize}}
\newcommand{\ben}{\begin{enumerate}}
\newcommand{\een}{\end{enumerate}}

\newcommand{\calA}{{\cal A}}
\newcommand{\calB}{{\cal B}}
\newcommand{\calC}{{\cal C}}
\newcommand{\calD}{{\cal D}}

\newcommand{\calH}{{\cal H}}

\newcommand{\calK}{{\cal K}}
\newcommand{\calL}{{\cal L}}

\newcommand{\calQ}{{\cal Q}}
\newcommand{\calR}{{\cal R}}
\newcommand{\calS}{{\cal S}}
\newcommand{\calT}{{\cal T}}

\newcommand{\calX}{{\cal X}}
\newcommand{\calY}{{\cal Y}}

\newcommand{\ie}{\emph{i.e.,}\xspace}

\newcommand{\hide}[1]{}

\newcommand{\hiddenText}{{\color{gray} hidden text.}}
\newcommand{\hideWithText}[1]{\hiddenText}

\newcommand{\subject}{\text{ subject to }}

\newcommand{\norm}[1]{\left\| #1 \right\|}

\newcommand{\diag}[1]{\mathrm{diag}\left(#1\right)}

\newcommand{\blue}[1]{{\color{blue}#1}}

\newcommand{\linkToPdf}[1]{\href{#1}{\blue{(pdf)}}}
\newcommand{\linkToPpt}[1]{\href{#1}{\blue{(ppt)}}}
\newcommand{\linkToCode}[1]{\href{#1}{\blue{(code)}}}
\newcommand{\linkToWeb}[1]{\href{#1}{\blue{(web)}}}
\newcommand{\linkToVideo}[1]{\href{#1}{\blue{(video)}}}
\newcommand{\linkToMedia}[1]{\href{#1}{\blue{(media)}}}
\newcommand{\award}[1]{\xspace} 

\newcommand{\R}{\mathbb{R}}

\renewcommand{\norm}[1]{\lVert #1 \rVert}
\newcommand{\inprod}[2]{\langle #1, #2 \rangle}

\newcommand{\sym}[1]{\mathbb{S}^{#1}}

\newcommand{\barcalA}{\bar{\calA}}

\newcommand{\bmat}{\left[ \begin{array}}
\newcommand{\emat}{\end{array}\right]}
\newcommand{\psd}[1]{\sym{#1}_{+}}

\newcommand{\abs}[1]{\left|#1\right|}

\newcommand{\KKT}{\mathrm{KKT}}
\newcommand{\normtwo}[1]{\norm{#1}_2}
\newcommand{\normF}[1]{\norm{#1}_\mathsf{F}}
\newcommand{\normHS}[1]{\norm{#1}_\mathrm{HS}}
\newcommand{\normD}[1]{\norm{#1}_\calD}
\newcommand{\normop}[1]{\norm{#1}_{\mathrm{op}}}

\newcommand{\zk}{z^{(k)}}
\newcommand{\zkpo}{z^{(k+1)}}
\newcommand{\xk}{x^{(k)}}
\newcommand{\xkpo}{x^{(k+1)}}
\newcommand{\sk}{s^{(k)}}
\newcommand{\skpo}{s^{(k+1)}}
\newcommand{\yk}{y^{(k)}}
\newcommand{\ykpo}{y^{(k+1)}}
\newcommand{\ykhalf}{y^{(k+1/2)}}

\newcommand{\PA}{P_{\calA}}
\newcommand{\PAp}{P_{\calA}^\perp}
\newcommand{\dist}{\mathrm{dist}}
\newcommand{\bbB}{\mathbb{B}}

\newcommand{\cl}[1]{\mathrm{cl}(#1)}
\newcommand{\Id}{\mathrm{Id}}
\newcommand{\Fix}{\mathrm{Fix}}
\newcommand{\ran}{\mathrm{ran}}
\newcommand{\tagmap}{\mathrm{tag}}
\newcommand{\fin}{\mathrm{fin}}
\newcommand{\inftag}{\mathrm{inf}}   

\newcommand{\Affp}{\mathrm{Aff}_+}

\newcommand{\eps}{\varepsilon}
\newcommand{\epsFamily}[1]{\{#1\}_{\varepsilon \downarrow 0}}
\newcommand{\bartheta}{\bar{\theta}}
\newcommand{\thetaEpsC}{\theta_\varepsilon^{\mathrm{c}}}
\newcommand{\thetaEpsP}{\theta_\varepsilon^{\mathrm{p}}}
\newcommand{\thetaEpsPOne}{\theta_\varepsilon^{\mathrm{p, 1}}}
\newcommand{\thetaEpsPTwo}{\theta_\varepsilon^{\mathrm{p, 2}}}
\newcommand{\TEpsC}{T_\varepsilon^{\mathrm{c}}}
\newcommand{\TEpsP}{T_\varepsilon^{\mathrm{p}}}

\newcommand{\deltaEpsC}{\delta_\varepsilon^{\mathrm{c}}}
\newcommand{\deltaEpsP}{\delta_\varepsilon^{\mathrm{p}}}
\newcommand{\deltaEpsPOne}{\delta_\varepsilon^{\mathrm{p, 1}}}
\newcommand{\deltaEpsPTwo}{\delta_\varepsilon^{\mathrm{p, 2}}}
\newcommand{\VEpsC}{V_\varepsilon^{\mathrm{c}}}
\newcommand{\VEpsP}{V_\varepsilon^{\mathrm{p}}}
\newcommand{\VEpsPOne}{V_\varepsilon^{\mathrm{p, 1}}}
\newcommand{\VEpsPTwo}{V_\varepsilon^{\mathrm{p, 2}}}

\newcommand{\vEpsP}{v_\varepsilon^{\mathrm{p}}}
\newcommand{\vEpsPOne}{v_\varepsilon^{\mathrm{p, 1}}}
\newcommand{\vEpsPTwo}{v_\varepsilon^{\mathrm{p, 2}}}
\newcommand{\SEpsOne}{\calS_\varepsilon^{1}}
\newcommand{\SEpsTwo}{\calS_\varepsilon^{2}}

\usepackage[capitalize,nameinlink,noabbrev]{cleveref}

\hypersetup{colorlinks, hypertexnames=false, pageanchor=true,
            linkcolor=blue, citecolor={green!50!black}, urlcolor=cyan,
            pdftitle={Afterimage Slow Regions in First-Order Methods for Linear Conic Programming}
            }
\mathtoolsset{centercolon}

\makeatletter
\AddToHook{cmd/appendix/before}{\def\cref@section@alias{appendix}}
\makeatother

\theoremstyle{plain}
\newtheorem{theorem}{Theorem}[section]
\newtheorem{proposition}[theorem]{Proposition}
\newtheorem{lemma}[theorem]{Lemma}

\newtheorem{definition}[theorem]{Definition}
\newtheorem{assumption}[theorem]{Assumption}
\theoremstyle{definition}

\theoremstyle{remark}

\newtheorem{example}[theorem]{Example}

\crefname{assumption}{Assumption}{Assumptions}
\crefname{problem}{Problem}{Problems}
\Crefname{assumption}{Assumption}{Assumptions}
\Crefname{problem}{Problem}{Problems}

\title{Afterimage Slow Regions in First-Order Methods \\ for Linear Conic Programming}

\author{
Shucheng Kang\thanks{School of Engineering and Applied Sciences, Harvard University.
Email: \texttt{skang1@g.harvard.edu}}
\and
Heng Yang\thanks{School of Engineering and Applied Sciences, Harvard University.
Email: \texttt{hankyang@seas.harvard.edu}}
}

\begin{document}

\maketitle

\begin{abstract}
First-order methods for linear conic programming often stall on long plateaus. Existing analyses characterize \emph{when} during a run or \emph{on which problem instances} slow convergence occurs; we instead ask \emph{where} in the state space slow convergence is present. We introduce the \emph{slow region family} for parameterized averaged fixed-point iterations. On a slow region, one step moves the state by only a small fraction of its distance to the fixed-point set, so an orbit starting there keeps almost its initial distance for arbitrarily many iterations. We then develop the \emph{afterimage} principle to construct them. A \emph{center} family and a nearby \emph{petal} family share the same limit parameter, so their residual fields become close, while their fixed-point sets or forward drifts stay far apart. The petal geometry then certifies a slow region for the center. We verify the standing assumptions for ADMM, sGS-ADMM, and PDHG, and give an LP, SOCP, and SDP gallery showing how varied slow regions can be.
\end{abstract}


\section{Introduction}
\label{sec:intro}

Linear conic programming (conic-LP), encompassing linear programming (LP),
second-order cone programming (SOCP), and semidefinite programming (SDP),
provides a versatile modeling framework for problems in machine learning
\citep{weinberger04cvpr-unsupervised-learning-sdp,lanckriet04jmlr-learn-kernal-matrix-sdp}, statistics \citep{d04nips-direct-pca-sdp},
and robust and combinatorial optimization
\citep{lobo98laa-applications-socp}. We consider the standard primal--dual pair:
\begin{align}
    \min_{x \in \calX} \inprod{c}{x}_\calX, & \quad \subject \calA x = b, \ x \in \calK, \tag{P} \label{eq:intro:conic-lp-primal} \\
    \max_{y \in \calY, s \in \calX} \inprod{b}{y}_\calY, & \quad \subject \calA^* y + s = c, \ s \in \calK^*, \tag{D} \label{eq:intro:conic-lp-dual}
\end{align}
where \(b \in \calY\), \(c \in \calX\), and \(\calA\) is a linear map from \(\calX\) to \(\calY\) with adjoint \(\calA^*\). Here \(\calX\) and \(\calY\) are finite-dimensional real Euclidean spaces (we identify \(\calY\) with \(\R^m\)), and \(\calK \subseteq \calX\) is a closed convex cone with dual cone \(\calK^*\). 

A popular route to solving conic-LP at scale is first-order methods (FOM)~\citep{ryu22book-monotone}, including the Alternating Direction Method of Multipliers (ADMM)~\citep{odonoghue16jota-scs,wen10mpc-admmsdp}, the Primal--Dual Hybrid Gradient method (PDHG)~\citep{chambolle11jmiv-pdhg,applegate21neurips-pdhg}, and symmetric Gauss--Seidel based ADMM (sGS-ADMM)~\citep{chen17mp-sgsadmm}. All of them run a dynamical system on a state space \(\calH\), driven by a fixed point map \(T_\theta: \calH \to \calH\):
\begin{align}
    \zkpo = T_\theta \zk, \label{eq:intro:fixed-point-update}
\end{align}
where \(T_\theta\) encodes the problem parameter \(\theta := (\calA, b, c)\), and the primal--dual variables \((x,y,s)\) are recovered from the state \(z \in \calH\) by a method-specific map. Under mild conditions \(\zk\) converges, as \(k \to \infty\), to a state representing a Karush--Kuhn--Tucker (KKT) point of the conic-LP. 

\paragraph{From slow phases to slow regions.} FOM convergence curves on conic-LP are rarely uniform: the KKT residual decreases quickly during an initial transient, and then nearly stalls for a long period. \emph{Temporal} explanations divide a run into stages. PDHG on LP first identifies the optimal basis and then converges linearly; the length of the first stage reflects how nearly degenerate the instance is, and the rate of the second is determined by a local sharpness constant~\citep{lu24mp-geometry}.
More generally, error-bound and metric-subregularity arguments give the same local picture for splitting methods near a sufficiently regular KKT point~\citep{han18mor-linear,kang25arxiv-admm}. \emph{Instance-geometric} explanations relate the whole complexity to the problem data, connecting restarted PDHG to the level-set geometry of the primal--dual pair~\citep{xiong24arxiv-levelset} and to condition measures that quantify how close an instance is to infeasibility or to multiple optima~\citep{xiong26mp-pdhg-lp-limiting-error-ratios,xiong26mor-accessible-complexity-bounds-rpdhg-lp}. Both answer \emph{when} in a run, or \emph{on which instances}, a method is slow. We instead ask \emph{where}:
\begin{quote}
    \emph{In which regions of the state space \(\calH\) does the iteration~\eqref{eq:intro:fixed-point-update} become very slow?}
\end{quote}
To answer this, we need a notion of slowness that belongs to a region rather than to a stage of a run. We take the cue from the local limit dynamics of ADMM on SDP~\citep{kang26arxiv-admmsdp-limitdyn}, where slowness is read from the limit residual field on the state space near a singular KKT point. There, slowness is a property of the region, not only of the instance: on one SDP, ADMM converges locally linearly near a strictly complementary KKT point~\citep{kang25arxiv-admm}, and nearly stalls near a singular one in the same optimal set. Such spatial diagnoses are common in numerical linear algebra: a few outlying small eigenvalues slow down conjugate gradient, and deflation removes that subspace to restore a fast rate~\citep{saad00sisc-deflated-cg}. What is still missing for FOM in conic-LP is a general way to construct and certify such regions, across different cones and different methods. This paper provides one. We start with a toy example that already contains the mechanism.

\paragraph{Motivating example.}
Consider~\eqref{eq:intro:conic-lp-primal} with \(\calX = \R^2\), \(\calK = \R^2_+\), \(\barcalA := [1, \ -1]\) and \(c := (1,0)\), and let only the right-hand side move:
\begin{align}
    \min_{x \in \R^2} \ x_1 \quad \subject x_1 - x_2 = \eps, \ x \ge 0, \qquad \eps \ge 0. \tag{LP\(_\eps\)} \label{eq:intro:lp-eps}
\end{align}
The primal optimum is the single point \((\eps,0)\) for every \(\eps \ge 0\). At \(\eps = 0\), however, the \emph{dual} optimal set degenerates into a whole segment. The ADMM map \(T_\theta\) inherits this degeneracy: at \(\eps = 0\) its fixed-point set is a segment \(\calL\), while for every \(\eps > 0\) it is a single point, which converges to one endpoint of \(\calL\) as \(\eps \downarrow 0\) (Figure~\ref{fig:intro:motivating}, left). In contrast, the residual field \(\delta_\theta(z) := T_\theta(z) - z\) shifts by the same vector at every state, of norm only \(\eps/\sqrt2\). As a result, if the iterates start from the other endpoint of \(\calL\), they move along \(\calL\) for about \(2\sigma/\eps\) steps, where \(\sigma\) is ADMM's penalty parameter, with \(\norm{\delta_\theta(\zk)}_2\) staying at \(\eps/\sqrt2\), before they detect the collapse. In words, a vanishing perturbation leaves a long-lived \emph{afterimage} of the unperturbed fixed-point set, and this afterimage is exactly where the iteration becomes slow.

\begin{figure}[htbp]
    \centering

    \begin{minipage}{\textwidth}
        \centering
        \begin{minipage}[b]{0.245\textwidth}
            \centering
            \includegraphics[width=\columnwidth]{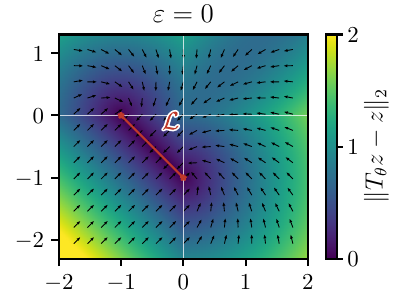}
        \end{minipage}
        \hfill
        \begin{minipage}[b]{0.245\textwidth}
            \centering
            \includegraphics[width=\columnwidth]{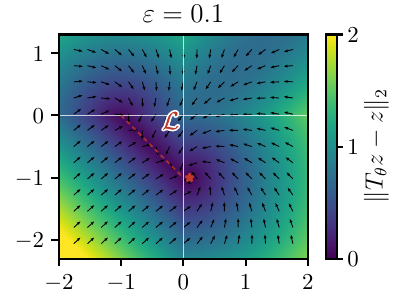}
        \end{minipage}
        \hfill
        \begin{minipage}[b]{0.245\textwidth}
            \centering
            \includegraphics[width=\columnwidth]{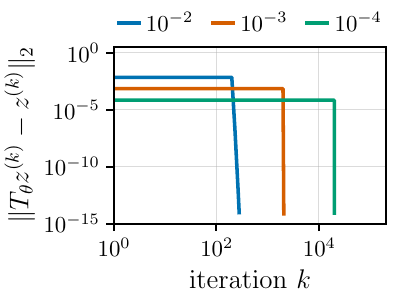}
        \end{minipage}
        \hfill
        \begin{minipage}[b]{0.245\textwidth}
            \centering
            \includegraphics[width=\columnwidth]{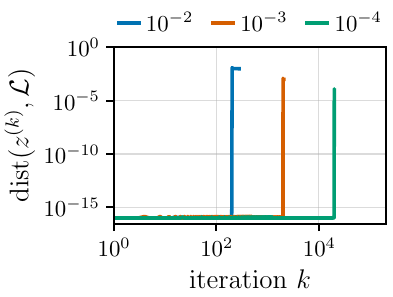}
        \end{minipage}
    \end{minipage}

    \caption{ADMM with \(\sigma = \tau = 1\). The motivating example under ADMM. \emph{Left two:} the residual field \(T_\theta z - z\) on \(\calH = \R^2\), colored by its norm, at \(\eps = 0\) and \(0.1\) --- nearly indistinguishable, yet \(\Fix (T_\theta)\) collapses from the segment \(\calL\) (solid; dashed on the right) to the single star. \emph{Right two:} from the far endpoint of \(\calL\), the same seed for every \(\eps\); each decade in \(\eps\) lengthens the plateau by a decade. \label{fig:intro:motivating}}
\end{figure}

\paragraph{Contributions.} We view the dynamical system~\eqref{eq:intro:fixed-point-update} as a \emph{static} residual field \(\delta_\theta\) on \(\calH\), and certify slow regions by comparing two different responses to a perturbation of \(\theta\): the residual field itself changes continuously, while the asymptotic geometry it carries --- its fixed points, or its forward drift --- can change abruptly.
Specifically, our contributions are threefold.

(\romannumeral1) We introduce the \emph{slow region family} for parameterized averaged fixed-point iterations under mild regularity conditions. Slowness is asymptotic and belongs to the whole family: each region is measured against the fixed-point set of its own parameter, and the family is slow if the ratio of the one-step residual to this distance tends to zero.
The regions may move with the parameter, or even diverge.
We also provide calculus rules for combining and thickening such families.

(\romannumeral2) We develop the \emph{afterimage} principle to construct such families. We compare a \emph{center} family, which carries the problems we actually solve, with a nearby \emph{petal} family. Both are indexed by a vanishing perturbation and converge to a common \emph{limit parameter}. Their residual fields therefore become close, while their fixed-point sets or forward drifts stay far apart, so that the petal geometry certifies a slow region family for the center iteration. Neither the petal nor the limiting problem needs to possess a KKT point: both may be solvable, (weakly) infeasible, non-attaining, or carry a positive duality gap. The principle covers a broad range of slow-convergence phenomena, including face selection, H\"older-sensitive displacement, and escape of fixed points to infinity, each with a bound on the resulting slowness. Figure~\ref{fig:intro:cover} summarizes the framework.

(\romannumeral3) We verify the standing assumptions for ADMM, sGS-ADMM, and PDHG under mild conditions.
We then present a gallery of LP, SOCP, and SDP examples, which exhibits residual fields, exact and approximate afterimages, and multiple certified slow regions, all matching the predicted vanishing moduli.

\begin{figure}[htbp]
    \centering
    \includegraphics[width=0.9\textwidth]{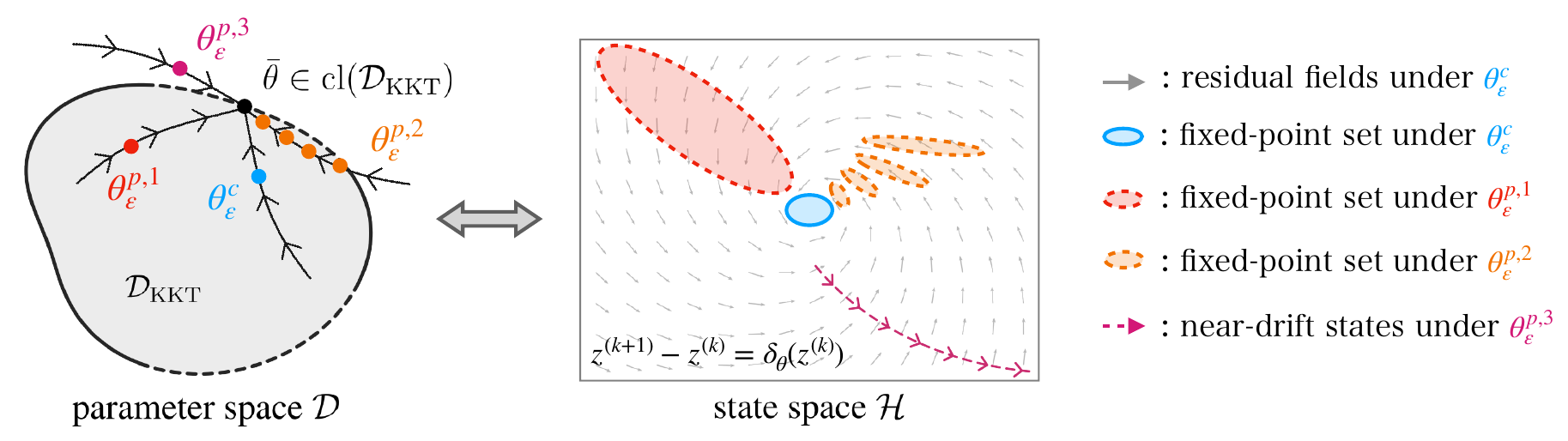}
    \caption{The afterimage framework. A \emph{center} family (the problems we actually solve) and several \emph{petal} families all converge to the same limit parameter \(\bartheta\), and induce nearly identical residual fields on bounded sets; their fixed-point sets and drift structures, however, differ sharply. Wherever a petal's structure is far from the center's, the center iterates make little progress toward their own fixed-point set for a long time, forming an \emph{afterimage slow region}. \label{fig:intro:cover}}
\end{figure}

\paragraph{Organization.} \S\ref{sec:asr} sets up the assumptions and develops the afterimage principle; \S\ref{sec:exp} presents the gallery; \S\ref{sec:conclusion} concludes. Related work is reviewed in \S\ref{app:related}.


\section{Afterimage Slow Regions in FOM for Conic-LP}
\label{sec:asr}

\subsection{Assumptions and notation}
\label{sec:asr:ass}

\paragraph{Notation.}
Throughout, \((\calX,\inprod{\cdot}{\cdot}_\calX)\), \((\calY,\inprod{\cdot}{\cdot}_\calY)\), \((\calH,\inprod{\cdot}{\cdot}_\calH)\) are finite-dimensional Hilbert spaces whose metrics are induced by the Euclidean inner product: \(\calX,\calY\) carry the conic-LP data of~\eqref{eq:intro:conic-lp-primal}--\eqref{eq:intro:conic-lp-dual}, namely \(c \in \calX\), \(b \in \calY\), \(\calA \in \calL(\calX,\calY)\) with adjoint \(\calA^*\), and a closed convex cone \(\calK \subseteq \calX\) with dual cone \(\calK^*\), while \(\calH\) carries the state \(z\) of~\eqref{eq:intro:fixed-point-update}. Here \(\calL(\calX,\calY)\) denotes the space of linear operators from \(\calX\) to \(\calY\) and \(\calL(\calH) := \calL(\calH,\calH)\); \(\normop{\cdot}\) and \(\normHS{\cdot}\) are the operator and Hilbert--Schmidt norms; \(\Id_\calH\), abbreviated \(\Id\) when unambiguous, is the identity on \(\calH\); and \(B \succeq_\calH A\), for self-adjoint \(A,B \in \calL(\calH)\), means \(\inprod{z}{(B-A)z}_\calH \ge 0\) for all \(z \in \calH\). If \(\calA\calA^* \succ_\calY 0\), we write \(\calA^\dagger := \calA^*(\calA\calA^*)^{-1} \in \calL(\calY,\calX)\) for the Moore--Penrose pseudo-inverse, \(\PA := \calA^\dagger\calA \in \calL(\calX)\) for the orthogonal projection onto \(\ran(\calA^*)\), and \(\PAp := \Id - \PA\). For a map \(T: \calH \to \calH\), \(\ran(T)\) and \(\Fix(T)\) are its range and fixed point set, and we may abbreviate \(T(z)\) as \(Tz\); \(\Affp := \{L: \R_+ \to \R_+ \mid L(r) = a + br \text{ for some } a,b \ge 0\}\) is the set of \emph{nonnegative affine moduli}. Given a norm \(\norm{\cdot}_\theta\) on \(\calH\) and a closed convex \(C \subseteq \calH\), \(\Pi_C^\theta\), \(\dist_\theta(\cdot, C)\), and \(\bbB_\theta(z,r)\) are the projection onto \(C\), the distance to \(C\), and the closed ball of radius \(r\) at \(z\), all in \(\norm{\cdot}_\theta\), with \(\Pi_C := \Pi_C^\calH\) for the Euclidean norm; \(\cl{C}\) is the closure of a set \(C\), and for a set family \(\epsFamily{C_\eps} \subset \calH\), \(C_\eps \to C\) as \(\eps \downarrow 0\) is meant in the Painlev\'e--Kuratowski sense. Finally, \([r] := \{1,\dots,r\}\), and the conic-LP pair is fully determined by the parameter \(\theta = (\calA,b,c) \in \calD := \calL(\calX,\calY) \times \calY \times \calX\), which we equip with the product norm \(\normD{\theta}^2 := \normHS{\calA}^2 + \norm{b}_\calY^2 + \norm{c}_\calX^2\).

\paragraph{Setup.} Define the set of KKT points as
\begin{align}
    \label{eq:asr:kkt}
    \KKT(\theta) := \{(x, y, s) \mid \calA x = b, \ \calA^* y + s = c, \ \inprod{x}{s}_\calX = 0, \ x \in \calK, \ s \in \calK^* \},
\end{align}
and we define \(\calD_{\KKT} := \{\theta \in \calD \mid \KKT(\theta) \ne \emptyset\}\). \(\calD_\KKT\) is not closed in general. 
We now formalize the assumptions imposed on the fixed point map in~\eqref{eq:intro:fixed-point-update}. Let \(\bartheta \in \cl{\calD_\KKT}\) and \(U \subset \calD\) be compact with \(\bartheta \in U\); \(U\) is not required to have interior. For every \(\theta \in U\), we associate a single-valued map \(T_\theta: \ \calH \to \calH\). Let \(\delta_\theta := T_\theta - \Id\) be \(T_\theta\)'s residual map (or \emph{residual field}, when viewed spatially)~\citep{kang26arxiv-admmsdp-limitdyn}. 

\begin{assumption}[Four standing assumptions]
    \label{ass:asr:T}
    The map family \(\{T_\theta\}_{\theta \in U}\) satisfies:

    (\romannumeral1) \(\Fix T_\theta \neq \emptyset\) iff \(\theta \in \calD_\KKT\) for every \(\theta\in U\).

    (\romannumeral2) For every \(\theta\in U\), \(T_\theta\) is \(\alpha_\theta\)-averaged with respect to the dynamic metric \(\inprod{\cdot}{\cdot}_\theta\) for some \(\alpha_\theta\in(0,1)\), where \(\inprod{z}{w}_\theta:=\inprod{z}{G_\theta w}_\calH\) and \(G_\theta\in\calL(\calH)\) is self-adjoint positive-definite.

    (\romannumeral3) There exist \(0<m_U\le M_U<\infty\) such that \(m_U\cdot \Id\preceq_\calH G_\theta\preceq_\calH M_U\cdot\Id\) for every \(\theta\in U\).

    (\romannumeral4) The map \(\calT:U\times\calH\to\calH\), \(\calT(\theta,z):=T_\theta z\), is Lipschitz continuous in the parameter, with a modulus \(L_U\in\Affp\):
    \begin{align}
        \label{eq:asr:ass-lip}
        \norm{T_\theta z-T_{\theta'}z}_\calH\le L_U(\norm{z}_\calH)\cdot\normD{\theta-\theta'},\quad\forall\,\theta,\theta'\in U,\ z\in\calH.
    \end{align}
\end{assumption}

Assumption~\ref{ass:asr:T} (\romannumeral2) is standard in monotone operator theory~\citep{ryu22book-monotone}, while (\romannumeral1) is method-specific and is verified for each map of Table~\ref{tab:fom} in Lemma~\ref{lem:app:fom:fix-kkt}. Assumption~\ref{ass:asr:T} (\romannumeral3) makes the dynamic norms uniformly equivalent to the ambient norm: \(\sqrt{m_U} \norm{z}_\calH \le \norm{z}_\theta \le \sqrt{M_U} \norm{z}_\calH\). Assumption~\ref{ass:asr:T} (\romannumeral4) states that nearby parameters induce uniformly close residual fields on any bounded state set,
\begin{align}
    \label{eq:asr:ass-lip-param}
    \sup_{z\in S}\norm{\delta_\theta(z)-\delta_{\theta'}(z)}_\calH\le L_U(r_S)\cdot\normD{\theta-\theta'},\qquad r_S:=\sup_{z\in S}\norm{z}_\calH,
\end{align}
for all \(\theta,\theta'\in U\) and all bounded \(S\subset\calH\), since \(\delta_\theta-\delta_{\theta'}=T_\theta-T_{\theta'}\). 


In Table~\ref{tab:fom}, we show that Assumption~\ref{ass:asr:T} holds for ADMM, PDHG, and sGS-ADMM under mild conditions. Appendix~\ref{app:fom} proves this and derives each method's modulus \(L_U\), omitted from the table. Several comments on Table~\ref{tab:fom}. First, the two additional conditions: \(\calA \calA^* \succeq_\calY \kappa_U \cdot \Id\) for ADMM/sGS-ADMM and \(\gamma_U < 1\) for PDHG are rather mild. The former is what makes the \(y\)-update~\eqref{eq:app:fom:admm-y-update} and the representation \(\calA^\dagger = \calA^*(\calA\calA^*)^{-1}\) well-defined; the latter is the standard condition making \(G_\theta \succ_\calH 0\), hence \(T_\theta\) averaged in that metric. Moreover, they do not need to hold simultaneously across different algorithms. Second, while the \(z\)-form of ADMM is the classical Douglas--Rachford reduction, we are not aware of an explicit statement that sGS-ADMM shares the same \(z\)-map. However, the two methods carry different primal--dual iterates \((x,y,s)\) along the same \(z\)-orbit, so the KKT residuals evaluated at their own iterates generally differ. Third, Assumption~\ref{ass:asr:T} does not cover every fixed point method. For instance, the exact Augmented Lagrangian Method (ALM) can fail (\romannumeral4). See Appendix~\ref{app:fom} for more details. 


\newcommand{\stk}[1]{\begin{tabular}[c]{@{}c@{}}#1\end{tabular}}

\begin{table}[t]
    \centering
    \setlength{\tabcolsep}{4pt}
    \renewcommand{\arraystretch}{1.4}
    \adjustbox{max width=\textwidth}{
    \begin{tabular}{|c|c|c|c|}
        \hline
        & \makebox[78pt]{ADMM} & \makebox[78pt]{sGS-ADMM} & PDHG \\
        \hline
        \stk{Conditions \\ on \(U\)}
        & \multicolumn{2}{c|}{\stk{\(\sigma > 0\), \(\tau \in (0,2)\), and \(\calA \calA^* \succeq_\calY \kappa_U \cdot \Id\) \\ for some \(\kappa_U > 0\) and all \(\theta \in U\)}}
        & \stk{\(\eta_x, \eta_y > 0\) and \\ \(\gamma_U := \sup_{\theta \in U} \sqrt{\eta_x \eta_y} \normop{\calA} < 1\)} \\
        \hline
        \(\calH\)
        & \multicolumn{2}{c|}{\(\calX\)}
        & \(\calX \times \calY\) \\
        \hline
        \(G_\theta\)
        & \multicolumn{2}{c|}{\(\PA + \tau^{-1} \PAp\)}
        & \(\begin{bmatrix} \eta_x^{-1} \cdot \Id_\calX & \calA^* \\ \calA & \eta_y^{-1} \cdot \Id_\calY \end{bmatrix}\) \\
        \hline
        \(\alpha_\theta\)
        & \multicolumn{2}{c|}{\(1/2 \cdot \max\{1, \tau\}\)}
        & \(1/2\) \\
        \hline
        \(m_U\)
        & \multicolumn{2}{c|}{\(\min\{1, \tau^{-1}\}\)}
        & \((1 - \gamma_U) \max\{\eta_x, \eta_y\}^{-1}\) \\
        \hline
        \(M_U\)
        & \multicolumn{2}{c|}{\(\max\{1, \tau^{-1}\}\)}
        & \((1 + \gamma_U) \min\{\eta_x, \eta_y\}^{-1}\) \\
        \hline
        \(T_\theta z\)
        & \multicolumn{2}{c|}{\stk{\(z - \PA [\Pi_{\calK}(z) - \calA^\dagger b]\) \\ \(- \, \tau \PAp [-\Pi_{\calK^*}(-z) + \sigma c]\)}}
        & \stk{\(\big(x^+, \ y + \eta_y [b - \calA (2 x^+ - x)]\big)\), \\ \(z = (x,y)\), \(x^+ := \Pi_\calK[x - \eta_x (c - \calA^* y)]\)} \\
        \hline
    \end{tabular}
    }
    \caption{Verification of Assumption~\ref{ass:asr:T} for three first-order methods. All entries are derived in Appendix~\ref{app:fom}. In their \(z\)-variable form, ADMM~\eqref{eq:app:fom:admm} and sGS-ADMM~\eqref{eq:app:fom:sgs} induce the \emph{same} fixed point map, hence share every entry above. Here \(\sigma\) is the penalty parameter and \(\tau\) is the multiplier step length of ADMM and sGS-ADMM, while \(\eta_x\) and \(\eta_y\) are the primal and the dual step sizes of PDHG; all of them are fixed over \(U\).}
    \label{tab:fom}
\end{table}

\paragraph{Partial superposition propositions.} In general the linear part of \(L_U(\cdot)\) cannot be \(0\): Appendix~\ref{app:fom} exhibits \(\theta, \theta' \in U\) along which \(\norm{T_\theta z - T_{\theta'} z}_\calH \to \infty\) as \(\norm{z}_\calH \to \infty\). Whenever it can be taken \(0\), however,~\eqref{eq:asr:ass-lip-param} becomes a state-\emph{uniform} bound, and in favorable cases the discrepancy is \emph{constant} in the state: for all \(\theta, \theta' \in U\),
\begin{align}
    \sup_{z \in \calH} \norm{\delta_\theta z - \delta_{\theta'}z}_\calH &\le l_U \cdot \normD{\theta - \theta'}, \label{eq:asr:ass-lip-strong} \\
    \delta_\theta z - \delta_{\theta'}z &= \calB(\theta - \theta'), \quad \forall\, z \in \calH, \label{eq:asr:ass-const-field}
\end{align}
where \(l_U\) depends only on \(U\), and \(\calB: \calD \to \calH\) is linear in the parameter and independent of the state; since \(\normop{\calB} < \infty\),~\eqref{eq:asr:ass-const-field} implies~\eqref{eq:asr:ass-lip-strong}. We call~\eqref{eq:asr:ass-lip-strong} and~\eqref{eq:asr:ass-const-field} the partial superposition propositions, after the Superposition Principle of electrostatic fields.

\begin{theorem}[Partial superposition propositions]
    \label{thm:asr:partial-superposition}
    Let Assumption~\ref{ass:asr:T} hold with \(U \subset \{\calA\} \times \calY \times \calX\), \ie \(\calA\) is fixed while \((b, c)\) may vary. Then (\romannumeral1)~\eqref{eq:asr:ass-const-field} holds for ADMM/sGS-ADMM; and (\romannumeral2)~\eqref{eq:asr:ass-lip-strong} holds for PDHG, upgrading to~\eqref{eq:asr:ass-const-field} when every \(\theta \in U\) also shares the same \(c\).
\end{theorem}
The proof is a direct computation from Table~\ref{tab:fom}, deferred to Appendix~\ref{app:fom}.

\subsection{Afterimage slow region family}
\label{sec:asr:asr}

Given \(\bartheta \in \cl{\calD_\KKT}\), a compact \(U \ni \bartheta\), and Assumption~\ref{ass:asr:T}, we take two families in \(U\), both indexed by \(\eps \downarrow 0\): the \emph{center} family \(\epsFamily{\thetaEpsC}\) and the \emph{petal} family \(\epsFamily{\thetaEpsP}\), with
\begin{align}
    \label{eq:asr:center-petal-family}
    \epsFamily{\thetaEpsC} \subset U \cap \calD_\KKT, \quad \epsFamily{\thetaEpsP} \subset U, \quad \thetaEpsC \to \bartheta, \ \thetaEpsP \to \bartheta,~\text{as}~\eps \downarrow 0.
\end{align}
The center family carries the conic-LP problems we actually study, while the petal family will certify slow regions for the center iteration. Only the center parameters are required to admit a KKT point; \(\bartheta\) and \(\thetaEpsP\) need not. We abbreviate \(T_{\thetaEpsC}, T_{\thetaEpsP}\) as \(\TEpsC, \TEpsP\), and likewise for the residual map \(\delta\), and write \(\VEpsC := \Fix(\TEpsC)\). Throughout, all conditions on \(\eps\)-indexed families are required to hold only for all sufficiently small \(\eps\).

\paragraph{Slow region family.} We define the slow region (SR) family of \(\epsFamily{\thetaEpsC}\) by a limiting process: 
\begin{definition}[Slow region family]
    \label{def:asr:sr-def}
    Suppose Assumption~\ref{ass:asr:T} holds. A set family \(\epsFamily{\calS_\eps} \subset \calH\) is called a slow region (SR) family of \(\epsFamily{\thetaEpsC}\), iff (\romannumeral1) \(\calS_\eps \ne \emptyset\) and \(\calS_\eps \cap \VEpsC = \emptyset\); (\romannumeral2) its \emph{vanishing modulus}
    \begin{align}
        \label{eq:asr:sr-def}
        \rho_\eps := \sup_{z \in \calS_\eps} \frac{\norm{\deltaEpsC(z)}_\calH}{\dist_\calH(z, \VEpsC)}
    \end{align} 
    is finite and \(\rho_\eps \to 0\) as \(\eps \downarrow 0\). 
\end{definition}
Equivalently, \(\rho_\eps\) is the least constant \(L\) with \(\norm{\deltaEpsC(z)}_\calH \le L \cdot \dist_\calH(z, \VEpsC)\) on \(\calS_\eps\); the best constant \(\kappa_\eps\) in the error bound \(\dist_\calH(z, \VEpsC) \le \kappa_\eps \norm{\deltaEpsC(z)}_\calH\) over \(\calS_\eps\) therefore satisfies \(\kappa_\eps \ge \rho_\eps^{-1}\), so an SR family is one on which this error-bound constant blows up; (\romannumeral1) already forces \(\rho_\eps > 0\). The definition is purely spatial and imposes no regularity on \(\epsFamily{\calS_\eps}\) or \(\epsFamily{\VEpsC}\); in particular, \(\calS_\eps\) collects starting states and need not be invariant: an orbit may leave \(\calS_\eps\) after one step. For the maps of Table~\ref{tab:fom}, the residual is also observable: \(\norm{\delta_\theta(z)}_\calH\) is equivalent to the projected KKT residual of a primal--dual pair computed from the state \(z\) (Lemma~\ref{lem:app:fom:kkt-bridge}).
Propositions~\ref{prop:asr:sr}--\ref{prop:asr:thickening} collect its basic properties. Due to page limits, all proofs in this subsection are deferred to Appendix~\ref{app:asr}. 
\begin{proposition}[Persistence away from the fixed point set]
    \label{prop:asr:sr}
    Under Definition~\ref{def:asr:sr-def}'s setting, fix \(\beta \in (0,1)\) and define \(N_\eps := \lfloor (1-\beta) \sqrt{m_U/M_U} \cdot \rho_\eps^{-1} \rfloor \to \infty\). Then:
    \begin{align}
        \label{eq:asr:dyn-slow}
        \dist_\calH((\TEpsC)^k (z), \VEpsC) \ge \sqrt{m_U/M_U} \beta \cdot \dist_\calH(z, \VEpsC), \quad \forall z \in \calS_\eps, \ 0 \le k \le N_\eps.
    \end{align}
\end{proposition}

\begin{proposition}[Finite unions]
    \label{prop:asr:finite-union}
    Let \(\epsFamily{\calS_\eps^i}\), \(i \in [r]\) with \(r \ge 2\), be SR families for the same \(\epsFamily{\thetaEpsC}\), with vanishing moduli \(\epsFamily{\rho_\eps^i}\). Then \(\epsFamily{\cup_{i \in [r]} \calS_\eps^i}\) is an SR family for \(\epsFamily{\thetaEpsC}\) with vanishing modulus exactly \(\epsFamily{\max_{i \in [r]}\rho_\eps^i}\).
\end{proposition}

Proposition~\ref{prop:asr:finite-union} does not extend verbatim to countably many families (Appendix~\ref{app:asr}).

\begin{proposition}[Thickening]
    \label{prop:asr:thickening}
    Under Definition~\ref{def:asr:sr-def}'s setting, let \(\epsFamily{a_\eps}\) be a real number family with \(0 \le a_\eps < 1\) and \(a_\eps \to 0\) as \(\eps \downarrow 0\). Define \(\tilde{\calS}_\eps := \{z \in \calH \mid \exists q \in \calS_\eps, \mathrm{~s.t.~} \norm{z - q}_\calH \le a_\eps \cdot \dist_\calH(q, \VEpsC) \}\). Then \(\epsFamily{\tilde{\calS}_\eps}\) is also an SR family of \(\epsFamily{\thetaEpsC}\), with vanishing modulus \(\tilde{\rho}_\eps \le (\rho_\eps + 2\sqrt{M_U/m_U} \cdot a_\eps) / (1 - a_\eps)\). 
\end{proposition}

\paragraph{Afterimage slow region family.} 
Finding an SR family directly is hard. However, if a petal family \(\epsFamily{\thetaEpsP}\) as in~\eqref{eq:asr:center-petal-family} has limiting behavior far from the center's, the SR family can often be certified. Before formalizing this idea, we record a fact from monotone operator theory, again proved in Appendix~\ref{app:asr}. 
\begin{proposition}[Forward drift]
    \label{prop:asr:forward-drift}
    Under Assumption~\ref{ass:asr:T}, take any \(\theta \in U\), and let \(\zkpo = T_\theta \zk\) from an arbitrary \(z^{(0)} \in \calH\). Then \(\delta_\theta \zk \to v_\theta := \Pi_{\cl{\ran(\delta_\theta)}}^\theta (0)\) as \(k \to \infty\), and \(\norm{\delta_\theta \zk}_\theta\) is nonincreasing in \(k\) with limit \(\norm{v_\theta}_\theta\). If \(\delta_\theta(z) = v_\theta\) for some \(z \in \calH\), then \(\delta_\theta(z + t v_\theta) = v_\theta\) for every \(t \ge 0\). Finally, if \(\Fix(T_\theta) = \emptyset\), there is a sequence \(\{z_j\} \subset \calH\) with \(\norm{z_j}_\calH \to \infty\) and \(\delta_\theta(z_j) \to v_\theta\). 
    We call \(v_\theta\) the \emph{forward drift}. 
\end{proposition}

Under Assumption~\ref{ass:asr:T}, \(\delta_\theta\) and \(v_\theta\) are well-defined whether or not \(\theta\) yields a KKT point. Define \(\tagmap(\theta) \in \{0,+\} \times \{\fin, \inftag\}\) as follows: the first entry is \(0\) iff \(v_\theta = 0\), and the second is \(\fin\) iff \(v_\theta \in \ran(\delta_\theta)\), \ie the projection is attained. The tag depends on the map family and its metric, not on \(\theta\) alone. Then \(\tagmap(\theta) = (0, \fin)\) iff \(\theta \in \calD_\KKT\); the other three cases collect the pathological parameters --- infeasible, weakly infeasible, nonattaining, or carrying a positive duality gap --- and all three are realized in \S\ref{sec:exp} and Appendix~\ref{app:exp}. 
For comparison, homogeneous self-dual embeddings recover primal--dual solutions or certificates of strong infeasibility, while weak infeasibility and other duality pathologies may lead to a degenerate embedding \citep{odonoghue16jota-scs,permenter17siopt-conic-hsde-facial}. 
The tags \(\tagmap(\bartheta)\), \(\tagmap(\thetaEpsC)\) and \(\tagmap(\thetaEpsP)\) may all differ. Abbreviating \(v_{\thetaEpsP}\) as \(\vEpsP\), we can now define the afterimage slow region (ASR) family. 

\begin{definition}[Afterimage slow region family]
    \label{def:asr:asr}
    Under Assumption~\ref{ass:asr:T}, let \(\epsFamily{\calS_\eps}\) be a slow region (SR) family for the center parameter family \(\epsFamily{\thetaEpsC}\). We call \(\epsFamily{\calS_\eps}\) an afterimage slow region (ASR) family w.r.t. a petal parameter family \(\epsFamily{\thetaEpsP}\), iff there exists a nonnegative real number family \(\epsFamily{\mu_\eps}\), with \(\mu_\eps \to 0\) as \(\eps \downarrow 0\), s.t. 
    \begin{align}
        \label{eq:asr:asr-def}
        \norm{\deltaEpsP(z) - \vEpsP}_\calH \le \mu_\eps, \quad \forall z \in \calS_\eps.
    \end{align} 
\end{definition}

Similar to \(\VEpsC = \Fix(\TEpsC)\), we let \(V_{\thetaEpsP} := \{z \in \calH \mid \deltaEpsP(z) = \vEpsP\}\), the fixed point set of the shifted map \(z \mapsto \TEpsP z - \vEpsP\), and abbreviate it as \(\VEpsP\). If \(\tagmap(\thetaEpsP)\) is \((0,\fin)\) or \((+,\fin)\) for all \(\eps\), then \(\VEpsP \ne \emptyset\); choosing \(\calS_\eps \subseteq \VEpsP\), as every example below with a finite petal tag does, we may then set \(\mu_\eps \equiv 0\). When it is \((0,\inftag)\) or \((+,\inftag)\), since \(\vEpsP\) is non-attainable and \(\VEpsP = \emptyset\), \(\mu_\eps\) must be positive. To see how the petal family helps to certify the ASR family, we strengthen the condition in~\eqref{eq:asr:sr-def} by asking for a family \(\epsFamily{\hat{\rho}_\eps}\) with \(\hat{\rho}_\eps \to 0\) and
\begin{align}
    \label{eq:asr:asr-strength}
    \hat{\rho}_\eps \cdot \dist_\calH(z, \VEpsC) \ge \norm{\deltaEpsP(z)}_\calH + \norm{\deltaEpsP(z) - \deltaEpsC(z)}_\calH, \quad \forall z \in \calS_\eps.
\end{align}
By the triangle inequality, \(\rho_\eps \le \hat{\rho}_\eps\); hence~\eqref{eq:asr:asr-strength} certifies~\eqref{eq:asr:sr-def} and exposes two residual field components we need to control. Intuitively, \(\norm{\deltaEpsP(z) - \deltaEpsC(z)}_\calH\) is controlled by~\eqref{eq:asr:ass-lip-param}, while~\eqref{eq:asr:asr-def} gives \(\norm{\deltaEpsP(z)}_\calH \le \norm{\vEpsP}_\calH + \mu_\eps\). We now illustrate Definition~\ref{def:asr:asr} through three special cases, drawn in Figure~\ref{fig:asr:three-cases}: \(\VEpsC\) collapses onto a proper face of \(\Fix(T_\bartheta)\), which is closed and convex (Lemma~\ref{lem:app:asr:fix-cc}); the petal fixed points are much farther from \(\Fix(T_\bartheta)\) than the data perturbation; and \(\VEpsC\) escapes to infinity. Each is only a sufficient condition: \S\ref{sec:exp} also constructs regions that none of the three covers, working directly from Definitions~\ref{def:asr:sr-def} and~\ref{def:asr:asr}. 

\begin{figure}[t]
    \centering
    \includegraphics[width=\textwidth]{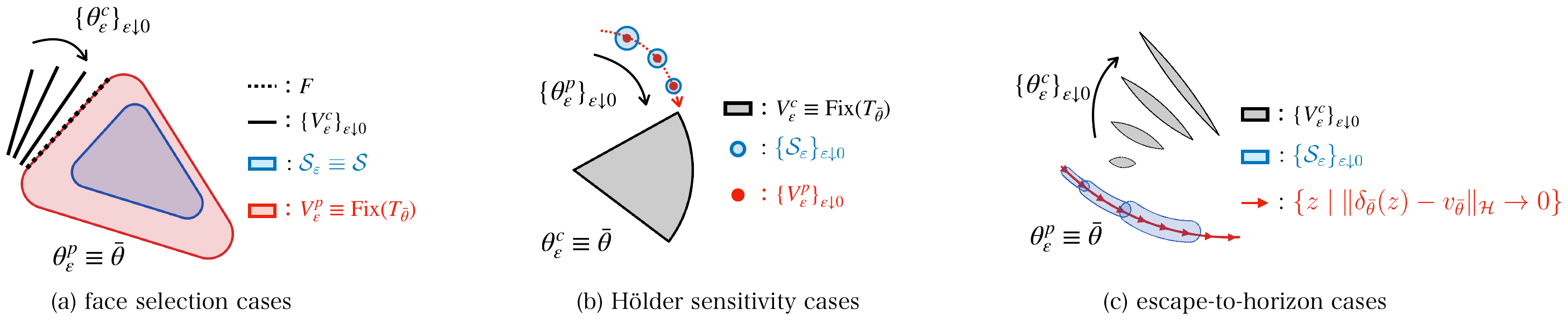}
    \caption{The three cases of Theorems~\ref{thm:asr:face-selection}--\ref{thm:asr:escape-to-horizon}, all drawn in the state space \(\calH\). \label{fig:asr:three-cases}}
\end{figure}

\begin{theorem}[Face selection cases]
    \label{thm:asr:face-selection}
    Under Assumption~\ref{ass:asr:T}, set \(\tagmap(\bartheta) = (0, \fin)\) and \(\thetaEpsP \equiv \bartheta\). Suppose \(\VEpsC \to F\) as \(\eps \downarrow 0\), where \(F\) is a proper face of \(\Fix(T_\bartheta)\), and take \(\calS_\eps \equiv \calS \subset \Fix(T_\bartheta)\) nonempty, bounded, with \(\inf_{z \in \calS}\dist_\calH(z, F) > 0\). Then \(\epsFamily{\calS_\eps}\) is an ASR family with \(\rho_\eps = O(\normD{\thetaEpsC - \bartheta})\).
\end{theorem}
\begin{proof}
    Painlev\'e--Kuratowski convergence \(\VEpsC \to F\) with \(F\) nonempty makes \(\dist_\calH(\cdot, \VEpsC) \to \dist_\calH(\cdot, F)\) uniformly on bounded sets~\citep[Theorem 4.35]{rockafellar98book-variational-analysis}. As \(\calS\) is bounded, \(d_\eps := \inf_{z \in \calS}\dist_\calH(z, \VEpsC) \to \inf_{z \in \calS}\dist_\calH(z, F) =: d_\star > 0\), so \(d_\eps > 0\) and \(\calS \cap \VEpsC = \emptyset\) for small \(\eps\). Put \(r_\calS := \sup_{z \in \calS}\norm{z}_\calH < \infty\). Since \(\thetaEpsP \equiv \bartheta\) and \(\calS \subset \Fix(T_\bartheta)\), \(\deltaEpsP \equiv 0\) on \(\calS\), giving~\eqref{eq:asr:asr-def} with \(\mu_\eps \equiv 0\); and~\eqref{eq:asr:ass-lip-param} bounds \(e_\eps := \sup_{z \in \calS}\norm{\deltaEpsC(z) - \deltaEpsP(z)}_\calH \le L_U(r_\calS) \cdot \normD{\thetaEpsC - \bartheta}\). Take \(\hat{\rho}_\eps := e_\eps / d_\eps\), so \(\rho_\eps \le \hat{\rho}_\eps\). Every \(z \in \calS\) then obeys \(\hat{\rho}_\eps \cdot \dist_\calH(z, \VEpsC) \ge e_\eps \ge \norm{\deltaEpsP(z)}_\calH + \norm{\deltaEpsP(z) - \deltaEpsC(z)}_\calH\), which is~\eqref{eq:asr:asr-strength}, and \(\rho_\eps = O(\normD{\thetaEpsC - \bartheta})\) since \(d_\eps \to d_\star > 0\).
\end{proof}
\begin{theorem}[H\"older sensitivity cases]
    \label{thm:asr:holder-sensitivity}
    Under Assumption~\ref{ass:asr:T}, set \(\tagmap(\bartheta) = (0, \fin)\) and \(\thetaEpsC \equiv \bartheta\). Suppose that for some \(\gamma \in (0,1)\) and all sufficiently small \(\eps\): \(\thetaEpsP \ne \bartheta\), \(\tagmap(\thetaEpsP) = (0, \fin)\), and there are \emph{witnesses} \(z_\eps \in \VEpsP\) with \(\sup_\eps \norm{z_\eps}_\calH < \infty\) and \(\dist_\calH(z_\eps, \Fix(T_\bartheta)) = \Theta(\normD{\thetaEpsP - \bartheta}^\gamma)\). Then there is an ASR family \(\epsFamily{\calS_\eps}\) with \(\rho_\eps = O(\normD{\thetaEpsP - \bartheta}^{1-\gamma})\); under~\eqref{eq:asr:ass-lip-strong} the witnesses need not be bounded.
\end{theorem}
\begin{proof}
    Write \(h_\eps := \normD{\thetaEpsP - \bartheta} > 0\) and \(\tilde{d}_\eps := \dist_\calH(z_\eps, \Fix(T_\bartheta))\), so that \(C_0 h_\eps^\gamma \ge \tilde{d}_\eps \ge c_0 h_\eps^\gamma > 0\) for some \(C_0 \ge c_0 > 0\). Here \(\VEpsC = \Fix(T_\bartheta)\) because \(\thetaEpsC \equiv \bartheta\), and \(\tagmap(\thetaEpsP) = (0,\fin)\) gives \(\vEpsP = 0\) with \(\VEpsP \ne \emptyset\). Fix \(\alpha \in (0,1)\) and set \(\calS_\eps := \VEpsP \cap \bbB_\calH(z_\eps, \alpha \tilde{d}_\eps)\), which contains \(z_\eps\). As \(\calS_\eps \subseteq \VEpsP\), \(\deltaEpsP \equiv 0\) there and~\eqref{eq:asr:asr-def} holds with \(\mu_\eps \equiv 0\). By \(1\)-Lipschitzness of \(\dist_\calH(\cdot, \VEpsC)\), \(\inf_{z \in \calS_\eps}\dist_\calH(z, \VEpsC) \ge (1-\alpha)\tilde{d}_\eps \ge c \cdot h_\eps^\gamma > 0\) with \(c := (1-\alpha)c_0\); in particular \(\calS_\eps \cap \VEpsC = \emptyset\). Moreover \(\sup_{z \in \calS_\eps}\norm{z}_\calH \le \sup_\eps \bigl(\norm{z_\eps}_\calH + \alpha C_0 h_\eps^\gamma\bigr) =: R < \infty\), where only the witnesses enter, so~\eqref{eq:asr:ass-lip-param} gives \(\sup_{z \in \calS_\eps}\norm{\deltaEpsP(z) - \deltaEpsC(z)}_\calH \le L_U(R) h_\eps\). Take \(\hat{\rho}_\eps := L_U(R)/c \cdot h_\eps^{1-\gamma}\), so \(\rho_\eps \le \hat{\rho}_\eps\), which tends to \(0\) as \(\gamma < 1\). Every \(z \in \calS_\eps\) obeys \(\hat{\rho}_\eps \cdot \dist_\calH(z, \VEpsC) \ge L_U(R) h_\eps \ge \norm{\deltaEpsP(z)}_\calH + \norm{\deltaEpsP(z) - \deltaEpsC(z)}_\calH\), which is~\eqref{eq:asr:asr-strength}. Under~\eqref{eq:asr:ass-lip-strong}, \(l_U\) replaces \(L_U(R)\) throughout and no bound on \(R\) is needed.
\end{proof}

\begin{theorem}[Escape-to-horizon cases]
    \label{thm:asr:escape-to-horizon}
    Under Assumption~\ref{ass:asr:T}, set \(\thetaEpsP \equiv \bartheta\) and suppose \(D_\eps := \inf_{z \in \VEpsC} \norm{z}_\calH \to \infty\) as \(\eps \downarrow 0\). Fix \(\gamma \in (0,1)\) and a nonnegative family \(\epsFamily{\nu_\eps}\) with \(\nu_\eps \to 0\). With \(g_{\bartheta}(D) := \min_{\norm{z}_\calH \le D} \norm{\delta_\bartheta(z) - v_\bartheta}_\calH\), let \(\calS_\eps := \{z \in \bbB_\calH(0, D_\eps^\gamma) \mid \norm{\delta_\bartheta(z) - v_\bartheta}_\calH \le g_{\bartheta}(D_\eps^\gamma) + \nu_\eps\}\). Then \(\epsFamily{\calS_\eps}\) is an ASR family with \(\rho_\eps = O(D_\eps^{-1} + D_\eps^{\gamma-1} \cdot \normD{\thetaEpsC - \bartheta})\).
\end{theorem}
\begin{proof}
    Here \(g_{\bartheta}\) is nonincreasing with \(g_{\bartheta}(D) \to 0\) as \(D \to \infty\): when \(\Fix(T_\bartheta) = \emptyset\), Proposition~\ref{prop:asr:forward-drift} supplies \(\{z_j\}\) with \(\norm{z_j}_\calH \to \infty\) and \(\delta_\bartheta(z_j) \to v_\bartheta\); otherwise \(v_\bartheta = 0\) and \(g_{\bartheta}\) vanishes on every ball meeting \(\Fix(T_\bartheta)\). The minimum defining \(g_{\bartheta}(D_\eps^\gamma)\) is attained, \(\bbB_\calH(0, D_\eps^\gamma)\) being compact and \(\delta_\bartheta\) continuous, and any minimizer lies in \(\calS_\eps\), so \(\calS_\eps \ne \emptyset\). Since \(\thetaEpsP \equiv \bartheta\) we have \(\deltaEpsP = \delta_\bartheta\) and \(\vEpsP = v_\bartheta\), so~\eqref{eq:asr:asr-def} holds with \(\mu_\eps := g_{\bartheta}(D_\eps^\gamma) + \nu_\eps \to 0\). Every \(z \in \calS_\eps\) has \(\norm{z}_\calH \le D_\eps^\gamma\) while every \(w \in \VEpsC\) has \(\norm{w}_\calH \ge D_\eps\), so \(\dist_\calH(z, \VEpsC) \ge D_\eps - D_\eps^\gamma > 0\) for small \(\eps\); in particular \(\calS_\eps \cap \VEpsC = \emptyset\). On \(\calS_\eps\), \(\norm{\deltaEpsP(z)}_\calH \le \norm{v_\bartheta}_\calH + \mu_\eps\), while~\eqref{eq:asr:ass-lip-param} gives \(\norm{\deltaEpsP(z) - \deltaEpsC(z)}_\calH \le L_U(D_\eps^\gamma) \cdot \normD{\thetaEpsC - \bartheta}\). Hence~\eqref{eq:asr:asr-strength} holds with
    \begin{align*}
        \hat{\rho}_\eps := \frac{
            L_U(D_\eps^\gamma) \cdot \normD{\thetaEpsC - \bartheta} + \norm{v_\bartheta}_\calH + \mu_\eps
        }{D_\eps - D_\eps^\gamma} = O(D_\eps^{\gamma - 1} \cdot \normD{\thetaEpsC - \bartheta} + D_\eps^{-1}).
    \end{align*}
    The order holds because \(L_U \in \Affp\) makes \(L_U(D_\eps^\gamma)/(D_\eps - D_\eps^\gamma) = O(D_\eps^{\gamma-1})\) while the rest of the numerator stays bounded. As \(\gamma < 1\), \(D_\eps \to \infty\) and \(\normD{\thetaEpsC - \bartheta} \to 0\), we get \(\rho_\eps \le \hat{\rho}_\eps \to 0\).
\end{proof}



\section{Afterimage Example Gallery}
\label{sec:exp}

We exhibit ASR families across all three problem classes --- LP, with \(\calX = \R^n\) and \(\calK = \R_+^n\); SOCP, with \(\calX = \R^n\) and \(\calK = \calQ^n := \{(t, v) : t \ge \normtwo{v}\}\); and SDP, with \(\calX = \sym{n}\) and \(\calK = \psd{n}\), where capital letters denote matrices --- under both ADMM/sGS-ADMM and PDHG. Example~\ref{exp:motivating} treats both maps and verifies Assumption~\ref{ass:asr:T}.


\begin{figure}[htbp]
    \centering

    \begin{minipage}{\textwidth}
        \centering
        \begin{minipage}[b]{0.245\textwidth}
            \centering
            \includegraphics[width=\columnwidth]{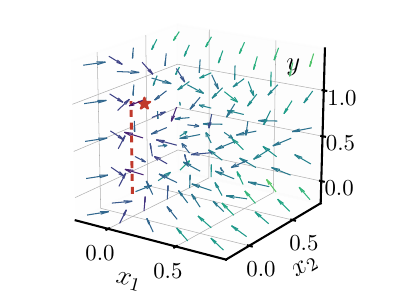}
        \end{minipage}
        \hfill
        \begin{minipage}[b]{0.245\textwidth}
            \centering
            \includegraphics[width=\columnwidth]{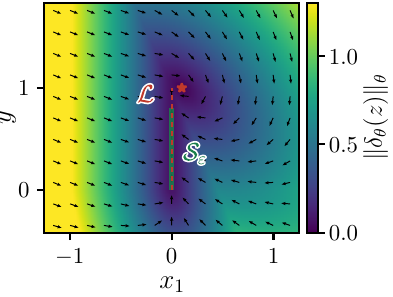}
        \end{minipage}
        \hfill
        \begin{minipage}[b]{0.245\textwidth}
            \centering
            \includegraphics[width=\columnwidth]{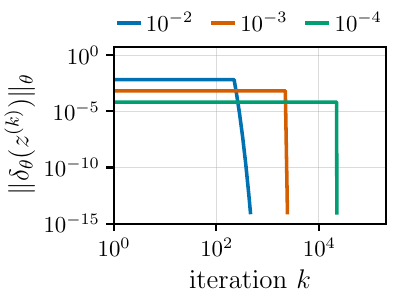}
        \end{minipage}
        \hfill
        \begin{minipage}[b]{0.245\textwidth}
            \centering
            \includegraphics[width=\columnwidth]{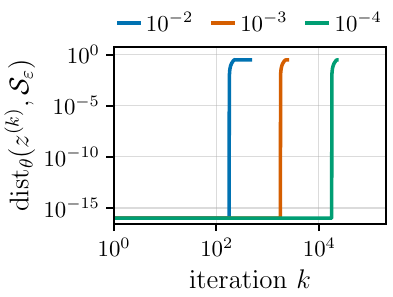}
        \end{minipage}
    \end{minipage}

    \caption{Example~\ref{exp:motivating} under PDHG, \(\eta_x = \eta_y = 0.45\). \emph{Left two:} \(\delta_\theta\) on \(\calH = \R^3\) at \(\eps = 0.1\), and its slice \(\{x_2 = 0\}\) carrying \(\calL = \Fix T_\bartheta\) (dashed), \(\VEpsC\) (star) and \(\calS\) (solid); arrows are in-slice components.
    \emph{Right two:} \(\norm{\delta_\theta(\zk)}_\theta\) and \(\dist_\theta(\zk,\calS)\) against \(k\), from one \(\eps\)-independent seed, at \(\eps = 10^{-2}, 10^{-3}, 10^{-4}\).\label{fig:exp:motivating-pdhg}}
\end{figure}

\begin{example}[Motivating example continued]
    \label{exp:motivating}
    Recall~\eqref{eq:intro:lp-eps}, in which only \(b\) moves: \(\bartheta := (\barcalA, 0, c)\), \(\thetaEpsC := (\barcalA, \eps, c)\), \(\thetaEpsP \equiv \bartheta\) and \(\normD{\thetaEpsC - \bartheta} = \eps\). Both tags equal \((0,\fin)\), and Assumption~\ref{ass:asr:T} holds on a small compact \(U \ni \bartheta\), since Table~\ref{tab:fom}'s side conditions are strict at \(\bartheta\) (Appendix~\ref{app:exp:motivating}). The dual degeneracy turns into a collapse of the fixed point sets: \(\calL := \Fix T_\bartheta\) is \(\{(-(1-t)\sigma,\,-t\sigma)\}_{t \in [0,1]}\) for ADMM and \(\{(0,0,y)\}_{y \in [0,1]}\) for PDHG, while \(\VEpsC\) is the single point \((\eps,-\sigma)\), resp.\ \((\eps,0,1)\), collapsing onto an endpoint \(F\). Only \(b\) moves, so~\eqref{eq:asr:ass-const-field} holds and Theorem~\ref{thm:asr:face-selection} gives, on \(\calS := \{z \in \calL : \dist_\calH(z,F) \ge c_0\}\) for any \(c_0 > 0\) below the length of \(\calL\), \(\mu_\eps \equiv 0\) and \(\rho_\eps {}\le \eps/(\sqrt2c_0)\), resp.\ \(\eta_y\eps/c_0\) --- the scalings displayed in Figures~\ref{fig:intro:motivating} and~\ref{fig:exp:motivating-pdhg}. Trimming the margin to \(c_\star\eps^\alpha\), any \(c_\star > 0\) below the length of \(\calL\) when \(\alpha = 0\), gives \(\rho_\eps = \Theta(\eps^{1-\alpha})\) for \(\alpha < 1\) and \(\Theta(1)\) beyond, an ASR family precisely for \(\alpha \in [0,1)\), certified directly in Appendix~\ref{app:exp:motivating} rather than by Theorem~\ref{thm:asr:face-selection}. Among \(\eps\)-independent singletons, \(\calL \setminus \{F\}\) is exactly the slow set, and fixed-parameter restart and Halpern schemes change only the constants of the \(\Theta(\eps^{-1})\) plateau (Appendix~\ref{app:exp:motivating}).
\end{example}

Each example below uses one of the two maps, takes Assumption~\ref{ass:asr:T} as given, and draws its residual field on a representative 2-D slice, with \(\sigma = \tau = 1\) for ADMM and \(\eta_x = \eta_y = 0.45\) for PDHG. Each contributes a feature the others lack. (\romannumeral1) Example~\ref{exp:sqrt} realizes Theorem~\ref{thm:asr:holder-sensitivity} at \(\gamma = 0.5\), where the residual and \(\rho_\eps\) scale by different powers of the perturbation. (\romannumeral2) Example~\ref{exp:two} attaches two petals to one center, and in the plotted run a single orbit passes slowly through both carriers in succession. (\romannumeral3) Example~\ref{exp:escape} has \(\Fix T_\bartheta = \emptyset\) because the limit is strongly infeasible, and the residual on its first region is independent of \(\eps\). Examples~\ref{exp:two} and~\ref{exp:escape} each carry a region outside the three cases of \S\ref{sec:asr:asr}. Appendix~\ref{app:exp} adds three more: (\romannumeral4) Example~\ref{exp:socp}, where the petal fixed set, the surviving face and the slow region are all positive-dimensional, with \(\VEpsC \subsetneq \VEpsP\) exact at every \(\eps > 0\); (\romannumeral5) Example~\ref{exp:nonatt}, whose drift is zero and unattained; and (\romannumeral6) Example~\ref{exp:adrift}, which loses attainment and zero drift together, since \(\calA(\calK)\) fails to be closed. The last two have \(\VEpsP = \emptyset\), so \(\mu_\eps > 0\) is forced. Two kinds of slowness appear below: on some regions the residual itself tends to zero, while in the escape examples it stays of order one and the slowness is relative to a fixed-point set that runs away.


\begin{figure}[htbp]
    \centering

    \begin{minipage}{\textwidth}
        \centering
        \begin{minipage}[b]{0.245\textwidth}
            \centering
            \includegraphics[width=\columnwidth]{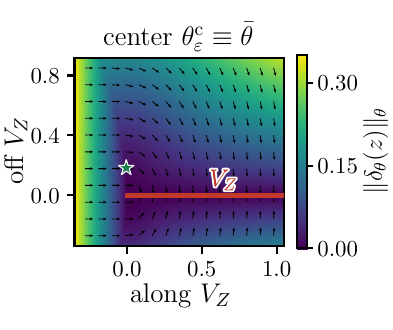}
        \end{minipage}
        \hfill
        \begin{minipage}[b]{0.245\textwidth}
            \centering
            \includegraphics[width=\columnwidth]{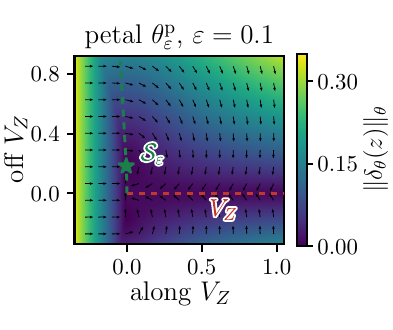}
        \end{minipage}
        \hfill
        \begin{minipage}[b]{0.245\textwidth}
            \centering
            \includegraphics[width=\columnwidth]{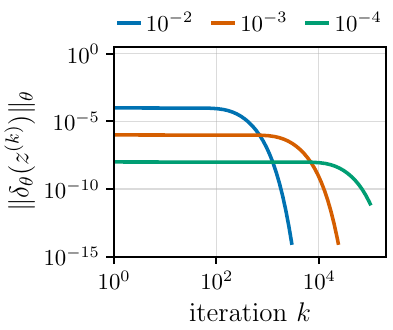}
        \end{minipage}
        \hfill
        \begin{minipage}[b]{0.245\textwidth}
            \centering
            \includegraphics[width=\columnwidth]{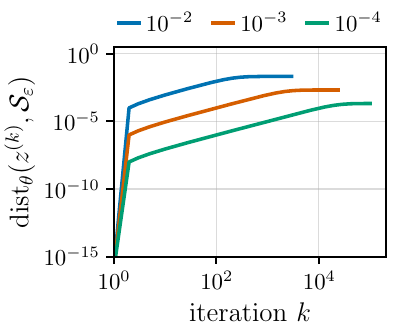}
        \end{minipage}
    \end{minipage}

    \caption{Example~\ref{exp:sqrt} under ADMM; layout as in Figure~\ref{fig:exp:motivating-pdhg}, on the plane through \(\VEpsC\) spanned by the ray and the dislocation; dashed green traces \(\VEpsP = \{Z_\eps\}\) (star) as \(\eps \downarrow 0\). Here the seed \(Z_\eps\) \emph{moves with \(\eps\)}: the residual sits near \(\eps^2\) while \(\dist_\theta(\zk,\calS_\eps)\) grows immediately, \(\calS_\eps\) being one point. \label{fig:exp:sqrt}}
\end{figure}

\begin{example}[SDP square-root H\"older]
    \label{exp:sqrt}
    We now \emph{freeze} the center and move the petal. Let \(\calK = \psd{2}\), \(A_1 := \bigl[\begin{smallmatrix}0&1\\1&-1\end{smallmatrix}\bigr]\), \(\barcalA(X) := \inprod{A_1}{X}_\calX\) and \(E_{ij}\) the matrix units: \(\bartheta = \thetaEpsC :\equiv (\barcalA,\, 0,\, E_{22})\), \(\thetaEpsP := (\barcalA,\, \eps^3,\, E_{22} + \eps^2E_{11})\) and \(h_\eps := \normD{\thetaEpsP - \bartheta} = \sqrt{\eps^6+\eps^4}\), with \(\tagmap(\bartheta) = (0,\fin)\). The center carries the ray \(\VEpsC = \Fix T_\bartheta = \{\diag{t,-1} : t \ge 0\}\) and the petal the single point \(\VEpsP = \{Z_\eps\}\), dislocated by \(\dist_\calH(Z_\eps,\VEpsC) = \sqrt3\,\eps + O(\eps^2) = \Theta(h_\eps^{1/2})\) (Appendix~\ref{app:exp:sqrt}). This is Theorem~\ref{thm:asr:holder-sensitivity} at \(\gamma = 0.5\): taking \(\calS_\eps = \VEpsP\), whose state-free defect has size \(\eps^2\sqrt{1+\eps^2/3}\), gives \(\rho_\eps = \Theta(\eps)\); see Figure~\ref{fig:exp:sqrt}. Here \(\calS_\eps\) is a single \emph{moving} point, so the seed depends on \(\eps\), and there is no invariant carrier to move along.
\end{example}


\begin{figure}[htbp]
    \centering

    \begin{minipage}{\textwidth}
        \centering
        \begin{minipage}[b]{0.245\textwidth}
            \centering
            \includegraphics[width=\columnwidth]{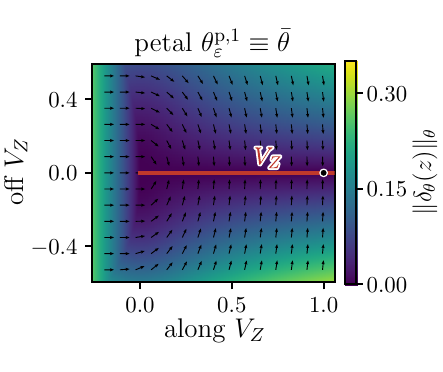}
        \end{minipage}
        \hfill
        \begin{minipage}[b]{0.245\textwidth}
            \centering
            \includegraphics[width=\columnwidth]{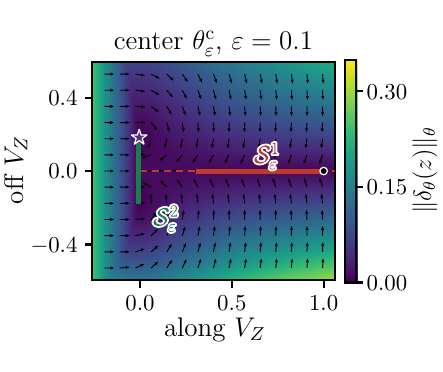}
        \end{minipage}
        \hfill
        \begin{minipage}[b]{0.245\textwidth}
            \centering
            \includegraphics[width=\columnwidth]{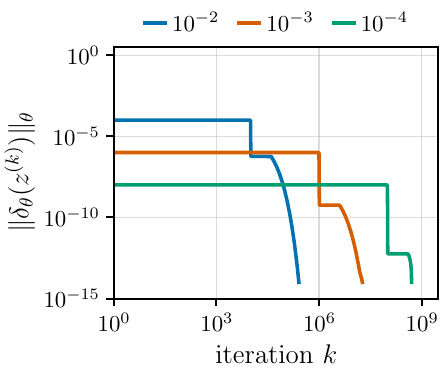}
        \end{minipage}
        \hfill
        \begin{minipage}[b]{0.245\textwidth}
            \centering
            \includegraphics[width=\columnwidth]{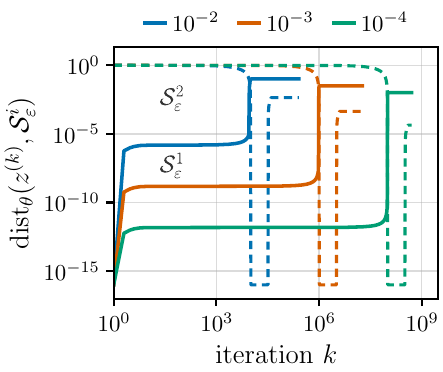}
        \end{minipage}
    \end{minipage}

    \caption{Example~\ref{exp:two} under ADMM; layout as in Figure~\ref{fig:exp:motivating-pdhg}. \emph{Left two:} the ray \(V_Z\), then its collapse to \(Z_\eps\) (star), both petal carriers surviving with trimmed parts solid. \emph{Right two:} one orbit from \(Q_{t_1}\), \(\alpha = 0.5\), \(\eta = 0.25\); solid and dashed are the distances to \(\SEpsOne\), \(\SEpsTwo\). \label{fig:exp:two}}
\end{figure}

\begin{example}[Two afterimages, two exponents]
    \label{exp:two}
    We keep the SDP data of Example~\ref{exp:sqrt} but read it \emph{backwards} --- what was the center there is a petal here --- and add a second petal. Writing \(C_\eps := E_{22} + \eps^2E_{11}\): \(\bartheta := (\barcalA,\, 0,\, E_{22})\), \(\thetaEpsC := (\barcalA,\, \eps^3,\, C_\eps)\), \(\thetaEpsPOne :\equiv \bartheta\) and \(\thetaEpsPTwo := (\barcalA,\, 0,\, C_\eps)\), with gaps \(\Theta(\eps^2)\) and \(\eps^3\) and \(\tagmap(\bartheta) = (0,\fin)\) as before. Three fixed sets are involved (Appendix~\ref{app:exp:two}): the ray \(\VEpsPOne = V_Z := \{Q_t := \diag{t,-1} : t \ge 0\}\); the point \(\VEpsC = \{Z_\eps\}\) of Example~\ref{exp:sqrt}, which collapses onto \(Q_0\) yet stays \(\sqrt3\eps + O(\eps^2)\) \emph{off} the ray; and the segment \(\VEpsPTwo = \{q(y) := -C_\eps + yA_1\}_{y \in [y_-,y_+]}\), with \(y_\pm\) the roots of the dual PSD condition \(y^2 = \eps^2(1+y)\). Both petals differ from the center in \(b\) and/or \(c\) only, so Theorem~\ref{thm:asr:partial-superposition}\,(\romannumeral1) makes both defects state-free; each \(\deltaEpsP\) vanishes on its own carrier, so \(\mu_\eps \equiv 0\).
    Fix \(t_1 \ge 1\), \(\alpha \in [0,1)\) and \(\eta \in (0,1)\): the regions \(\SEpsOne := \{Q_t : \eps^\alpha \le t \le t_1\}\) and \(\SEpsTwo := \{q(y) : y_- \le y \le (1-\eta)y_+\}\) have \(\rho^1_\eps = \Theta(\eps^{2-\alpha})\) and \(\rho^2_\eps = \Theta(\eps^2)\). Region~\(1\) is face selection with a tunable trim, certified directly in Appendix~\ref{app:exp:two}.
    Region~\(2\) falls outside \S\ref{sec:asr:asr}: \(\thetaEpsC\) and \(\thetaEpsPTwo\) both move and neither equals \(\bartheta\), so Definitions~\ref{def:asr:sr-def} and~\ref{def:asr:asr} certify it directly. One center thus carries two afterimages whose \emph{residuals} differ by a full order of \(\eps\); see Figure~\ref{fig:exp:two}.
\end{example}


\begin{figure}[htbp]
    \centering

    \begin{minipage}{\textwidth}
        \centering
        \begin{minipage}[b]{0.245\textwidth}
            \centering
            \includegraphics[width=\columnwidth]{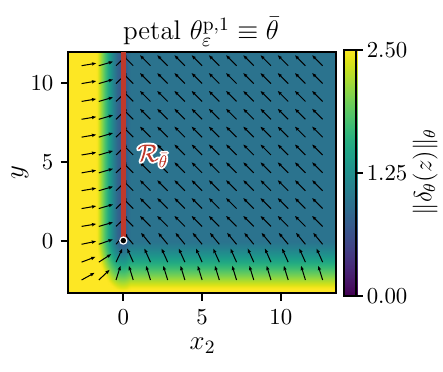}
        \end{minipage}
        \hfill
        \begin{minipage}[b]{0.245\textwidth}
            \centering
            \includegraphics[width=\columnwidth]{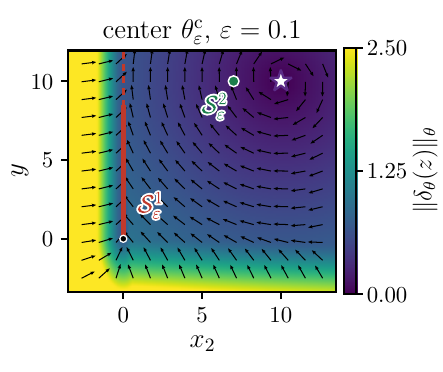}
        \end{minipage}
        \hfill
        \begin{minipage}[b]{0.245\textwidth}
            \centering
            \includegraphics[width=\columnwidth]{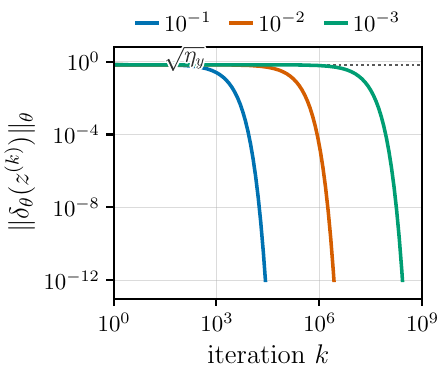}
        \end{minipage}
        \hfill
        \begin{minipage}[b]{0.245\textwidth}
            \centering
            \includegraphics[width=\columnwidth]{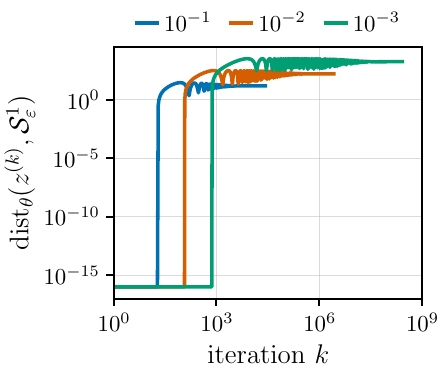}
        \end{minipage}
    \end{minipage}

    \caption{Example~\ref{exp:escape} under PDHG; layout as in Figure~\ref{fig:exp:motivating-pdhg}, on the slice \(\{x_1 = 0\}\). \emph{Left two:} with no fixed point at the limit the color never reaches \(0\), so \(\calR_\bartheta\) is drawn instead; then one appears (star) and escapes, \(\SEpsOne\) solid and \(\SEpsTwo\) the green dot. \emph{Right two:} one orbit from \(z^{(0)} = 0\), \(\gamma = 0.8\); the three plateaus share the same height \(\sqrt{\eta_y}\) and only lengthen. \label{fig:exp:escape}}
\end{figure}

\begin{example}[Escape to the horizon]
    \label{exp:escape}
    Let \(\calK = \R^2_+\), \(c := (0,1)\), \(a := 1/\eps\) and \(0 < \xi_\eps \to 0\): \(\bartheta := ([-1\ \ 0],\, 1,\, c)\), \(\thetaEpsC := ([-1\ \ \eps],\, 1,\, c)\), \(\thetaEpsPOne :\equiv \bartheta\) and \(\thetaEpsPTwo := ([-1\ \ \eps],\, 1 - \xi_\eps,\, c)\), with gaps \(\eps\) and \(\xi_\eps\). The limit is strongly infeasible, so \(\Fix T_\bartheta = \emptyset\), while its drift \(v_\bartheta = (0,0,\eta_y)\) is attained on the ray \(\calR_\bartheta := \{(0,0,y) : y \ge 0\} = \VEpsPOne\): \(\tagmap(\bartheta) = (+,\fin)\). The center is solvable, but its unique fixed point \(\VEpsC = \{(0,a,a)\}\) \emph{escapes}, \(D_\eps := \inf_{z \in \VEpsC}\norm{z}_\calH = \sqrt2/\eps \to \infty\). For any \(\gamma \in (0,1)\), the regions \(\SEpsOne := \calR_\bartheta \cap \bbB_\calH(0,D_\eps^\gamma)\) and \(\SEpsTwo := \VEpsPTwo = \{(0,(1-\xi_\eps)a,\,a)\}\) have \(\mu_\eps \equiv 0\), \(\rho^1_\eps = \Theta(\eps)\) and \(\rho^2_\eps = \eta_y\eps\) exactly (Appendix~\ref{app:exp:escape}). On \(\SEpsOne\) the residual equals \(\norm{v_\bartheta}_\theta = \sqrt{\eta_y}\), independent of \(\eps\); Region~\(2\) again falls outside \S\ref{sec:asr:asr}, since both \(\thetaEpsC\) and \(\thetaEpsPTwo\) move and neither equals \(\bartheta\).
\end{example}


\section{Conclusion}
\label{sec:conclusion}
We introduced a framework that certifies slow regions in first-order methods for conic-LP: a region family is slow once the residual vanishes relative to the distance to the fixed-point set of its own parameter. The afterimage principle explains where such families come from: a vanishing perturbation moves the residual field continuously, while the geometry it carries can change abruptly, so that the geometry of a neighboring problem survives as a slow region of the problem we actually solve. The standing assumptions hold for ADMM, sGS-ADMM, and PDHG. Moreover, the mechanism requires neither a polyhedral cone nor a well-posed limit, so the gallery covers infeasible and non-attaining problems. Our framework also has two limitations. First, it is purely spatial: it measures the strength of the residual field on a region, but says nothing about the dynamics --- whether typical orbits enter a slow region, or for how long they stay. Second, every certificate starts from a petal family designed by hand; we do not yet know how to find one for a given, non-designed instance. Two directions are worth exploring in future work: locating afterimage slow regions in specific conic-LP families, and designing acceleration schemes around them.

\clearpage

\section*{Acknowledgments}
This project is partly funded by the NSF CAREER Award 2543352.


\section*{AI use statement}
\label{sec:llm}
The framework of this paper --- its assumptions, definitions, theorems, and the example gallery with its
derivations --- was developed and written by the authors. Large language models (LLMs) were used to copy-edit and
compress the exposition throughout, and to review the draft adversarially.
The authors
take full responsibility for this paper.

\bibliographystyle{plainnat}
\bibliography{../../references/refs,../../references/myRefs}

\appendix

\section{Appendix}
\label{app}


\subsection{Related Work}
\label{app:related}

Our afterimage slow region framework is related to three lines of work: recent solver developments for FOM on conic-LP, convergence theory of FOM on conic-LP, and dynamical system theory.

\paragraph{Recent solver developments in FOM for conic-LP.} ADMM, sGS-ADMM and PDHG have become popular choices for large-scale conic-LP, and a number of solvers are built upon them~\citep{odonoghue16jota-scs,zheng17ifac-cdcs-sdpsolver,garstka21jota-cosmo}. A related family of methods maintains a second-order outer loop without ever forming the full Newton system. Semismooth Newton (SSN) is applied either to the Augmented Lagrangian Method (ALM) subproblem~\citep{zhao10siopt-newtoncg-alm-sdp,yang15mpc-sdpnalplus-sdpsolver} or directly to the fixed-point residual of~\eqref{eq:intro:fixed-point-update}~\citep{li18sisc-ssn-sdp,deng25arxiv-efficient-ssn-sdp}, with the Newton system solved inexactly by conjugate gradient. Alternatively, a low-rank factorization reduces the size of the subproblem, which is then handled by Riemannian methods on the resulting manifold~\citep{tang24siopt-feasible-lowrank-sdp,hou26or-rnnal} or by a hybrid convex--nonconvex scheme~\citep{monteiro26mp-hallar}. Like FOM, these methods remain sensitive to conditioning. More recently, GPU implementations of FOM for conic-LP have appeared~\citep{lu25or-cupdlp,lin25arxiv-pdcs,han25ijoc-low-rank,groudiev25-cuadmm,aguirre25arxiv-cuhallar,ding25arxiv-new-understanding-sdp-lowrank}.

\paragraph{Recent theoretical developments in FOM for conic-LP.} As general splitting schemes, FOM have well-established worst-case sublinear rates in monotone operator theory~\citep{ryu22book-monotone}, provided the fixed point set is nonempty.
On conic-LP, however, local linear convergence is observed much more frequently than the worst-case rate suggests, even on degenerate instances. For LP this is classical: global sharpness follows from polyhedral error bounds~\citep{hoffman03ws-approximate-solution-linear-inequalities,mangasarian79sjco-nonlinear-perturbation-lp}, and later work has further refined it for PDHG~\citep{applegate21neurips-pdhg,lu24mp-geometry}. For ADMM on LP/QP, \citet{boley13siopt-linearconv-admm-lp-qp} identifies constant-step transients before the local linear phase.
General conic-LP is much harder. Metric subregularity of the KKT operator, which would imply a local linear rate, can fail~\citep{cui16arxiv-superlinear-alm-sdp}. \citet{han18mor-linear} recover it under uniqueness of the KKT point together with a two-sided second-order regularity condition. For a non-unique KKT point set,~\citet{kang25arxiv-admm} prove local linear convergence of ADMM on SDP whenever the limit point satisfies strict complementarity, and further generalizations have followed~\citep{jiang26arxiv-linearconv-pdhg-sdp,li26arxiv-gpu-conic-qp-linearconv-sc,ding26arxiv-linearconv-fom-cp,wang25arxiv-dual-riemannian-admm-lowrank-sdp}. These analyses are anchored at a KKT point of the instance being solved; the afterimage certificate is instead anchored at a neighboring problem. Another line works on the pathological instance itself, using the minimal displacement vector of DRS/ADMM and PDHG as an infeasibility certificate~\citep{banjac21orl-minimal-displacement-vector-drs,liu17arxiv-new-drs-pathological-conic-lp,jiang23arxiv-range-displacement-pdhg,applegate24siopt-infeasibility-detection-pdhg-lp}; the afterimage instead uses a \emph{neighboring} problem's drift to certify slow states of a solvable center.
A separate line of work measures an instance's distance from ill-posedness. For a feasible conic system, the size of the smallest data perturbation that destroys feasibility governs both the geometry of the feasible region and the cost of solving it~\citep{renegar94mp-perturbation-lp,freund99mp-distance-ill-conic-linear-system}; more recent works on restarted PDHG map an LP or conic instance to a condition measure, and then to an iteration count~\citep{xiong24arxiv-levelset,xiong26mp-pdhg-lp-limiting-error-ratios,xiong26mor-accessible-complexity-bounds-rpdhg-lp}. We instead locate the slow states inside \(\calH\), allowing non-polyhedral cones and ill-posed limits. Closest to us, \citet{bellon25mor-parametric} classify how the optimal set of a parametric SDP degenerates; we study how such degeneration affects the iterates.

\paragraph{Dynamical system theory.} Locating slow transients inside a convergent dynamical system is a classical difficulty. Bifurcation theory, especially center manifold theory, isolates slow motion by linearizing at an equilibrium and splitting the spectrum~\citep{guckenheimer13ssba-dynamical-system-bifurcations,kuznetsov98springer-elements-applied-bifurcation}.
For averaged nonsmooth maps, however, a smooth linearization to split is generally unavailable, and there is no distinguished critical direction. We therefore certify a transient region instead of an invariant manifold, and this region need not contain a fixed point.
Many algorithms in numerical analysis can likewise be viewed as dynamical systems. Pseudospectral theory, for instance, explains the transient behavior of non-normal operators and of the eigenvalue algorithms built on them~\citep{trefethen05pup-pseudospectra}. Within optimization,~\citet{poon19nips-trajectory-admm,poon20arxiv-geometry-fom} study the trajectory and geometry of FOM; for LP,~\citet{liu24arxiv-crossover-pdlp} identify the spiral trajectories of PDHG within a fixed basis; and for ADMM on degenerate SDP,~\citet{kang26arxiv-admmsdp-limitdyn} give a second-order spatial analysis of the residual field and of the slow regions surrounding a singular KKT point. These works motivate us to look for slow regions in FOM for conic-LP in general.

\subsection{Additional material for \S\ref{sec:asr:ass}}
\label{app:fom}

In this appendix, we prove Table~\ref{tab:fom} by verifying Assumption~\ref{ass:asr:T} for ADMM, sGS-ADMM, and PDHG.

\paragraph{Why exact ALM is not covered.} 
Exact ALM is the proximal point method applied to the dual~\citep{rockafellar76sicon-monotone}, with state \(y \in \calY\) and \(T_\theta(y) = \Pi_{D_\theta}(y + \sigma b)\), where \(D_\theta := \{y \in \calY \mid c - \calA^*y \in \calK^*\}\) itself carries \(\theta\) and need not vary continuously with it. Take \(\calX = \R^2\), \(\calK = \R^2_+\), \(\calY = \R\), and \(\theta_\eps := (\calA_\eps, 0, (1,0))\) with \(\calA_\eps := [1\ \ \eps]\), \(\eps \ge 0\). Every \(\theta_\eps\) is a solvable LP of optimal value \(0\) with attained dual; the rank condition of Table~\ref{tab:fom} holds uniformly. Since \(c - \calA_\eps^*y = (1-y,\ -\eps y)\), we get \(D_{\theta_\eps} = (-\infty,0]\) for every \(\eps > 0\) while \(D_{\theta_0} = (-\infty,1]\): the dual feasible set collapses discontinuously in the limit. At the single state \(y = 0.5\), \(T_{\theta_\eps}(0.5) = 0\) but \(T_{\theta_0}(0.5) = 0.5\), yet \(\normD{\theta_\eps - \theta_0} = \eps\); Assumption~\ref{ass:asr:T}~(\romannumeral4) would demand \(0.5 \le L_U(0.5)\,\eps\) for all small \(\eps\), which no \(L_U \in \Affp\) satisfies.


\paragraph{ADMM case.}
Fix \(\sigma > 0\) and \(\tau \in (0, 2)\). Suppose \(\calA \calA^* \succeq_\calY \kappa_U \cdot \Id\) with constant \(\kappa_U > 0\). Consider the classical three-step ADMM with multiplier step length \(\tau\): 
\begin{subequations}
    \label{eq:app:fom:admm}
    \begin{align}
        \ykpo & = (\calA \calA^*)^{-1} [\sigma^{-1} b - \calA (\sigma^{-1} \xk + \sk - c)], \label{eq:app:fom:admm-y-update} \\
        \skpo & = \Pi_{\calK^*}(c - \calA^* \ykpo - \sigma^{-1} \xk), \label{eq:app:fom:admm-s-update} \\
        \xkpo & = \xk +\tau \sigma (\calA^* \ykpo + \skpo - c).  \label{eq:app:fom:admm-x-update}
    \end{align}
\end{subequations}
For \(k \ge 0\), define the auxiliary state variable \(\zk := \PAp \xk + \calA^\dagger b - \sigma \PA \sk - \sigma \PAp c\). Then, from~\eqref{eq:app:fom:admm-s-update},
\(
    \skpo = \Pi_{\calK^*}[c - \sigma^{-1} \calA^\dagger b + \PA (\sigma^{-1}\xk + \sk -c) - \sigma^{-1} \xk] = \sigma^{-1} \Pi_{\calK^*}(-\zk). 
\)
Furthermore, from~\eqref{eq:app:fom:admm-x-update}, 
\begin{align*}
    & \xkpo = \xk + \tau [\calA^\dagger b - \PA (\xk + \sigma \sk - \sigma c) + \Pi_{\calK^*}(-\zk) - \sigma c] \\
    = & (1 - \tau) \xk + \tau \Pi_{\calK^*}(-\zk) + \tau (\PAp \xk - \sigma \PA \sk + \calA^\dagger b - \sigma \PAp c) \\
    = & (1 - \tau) \xk + \tau \Pi_{\calK^*}(-\zk) + \tau \zk = (1 - \tau) \xk + \tau \Pi_{\calK}(\zk),
\end{align*}
where we use Moreau identity \(z = \Pi_\calK(z) - \Pi_{\calK^*}(-z)\) in the last equality. Put \(\xkpo\) and \(\skpo\)'s new formula into \(\zkpo = \PAp \xkpo + \calA^\dagger b - \sigma \PA \skpo - \sigma \PAp c\), and with the observation that \(\PAp \zk = \PAp \xk - \sigma \PAp c\) for all \(k \ge 0\), we get 
\begin{align}
    & \zkpo = (1 - \tau) \PAp \xk + \tau \PAp \Pi_\calK(\zk) + \calA^\dagger b - \PA \Pi_{\calK^*}(-\zk) - \sigma \PAp c \nonumber \\
    = & (1-\tau) \PAp( \zk + \sigma c) + \tau \PAp [\zk + \Pi_{\calK^*}(-\zk)] + \calA^\dagger b + \PA [\zk - \Pi_{\calK}(\zk)] - \sigma \PAp c \nonumber \\
    = & \zk - \PA [\Pi_{\calK}(\zk) - \calA^\dagger b] - \tau \PAp [-\Pi_{\calK^*}(-\zk) + \sigma c], \label{eq:app:fom:admm-Ttheta} 
\end{align}
which is exactly \(T_\theta(\zk)\)'s form. To recover the primal--dual variable iterates from \(\zk\)'s: 
\begin{align}
    \label{eq:app:fom:admm-recover}
    \xkpo = (1-\tau) \xk + \tau \Pi_{\calK}(\zk), \quad \skpo = \sigma^{-1} \Pi_{\calK^*}(-\zk),
\end{align}
and \(\ykpo\) from~\eqref{eq:app:fom:admm-y-update} for \(k \ge 0\). 

The admissible multiplier step-length range \(\tau \in (0,2)\) is known for two-block ADMM when one objective block is linear; see~\citep{gabay76cma-dual-algorithm-nonlinear-variation} and, for full-sequence convergence to the set of KKT points,~\citep{chen21mp-equivalence-proximal-alm-admm}. Here we also provide a short proof of its averagedness. Let \(G_\theta := \PA + \tau^{-1} \PAp\). Clearly, 
\begin{align}
    \label{eq:app:fom:admm-Gtheta}
    \min\{1,\tau^{-1}\} \cdot \Id \preceq_\calX G_\theta \preceq_\calX \max\{1,\tau^{-1} \}\cdot \Id.
\end{align}
Abbreviate \(\PA [\Pi_{\calK}(z) - \calA^\dagger b]\) as \(r^\parallel_\theta(z)\) and \(\PAp [-\Pi_{\calK^*}(-z) + \sigma c]\) as \(r^\perp_\theta(z)\), respectively. Then for ADMM, \(T_\theta(z) = z - r^\parallel_\theta(z) - \tau r^\perp_\theta(z)\) from~\eqref{eq:app:fom:admm-Ttheta}. For \(\tau = 1\), it is known from Douglas--Rachford splitting that \(T_\theta\) is firmly nonexpansive (\ie 1/2-averaged) under \(\inprod{\cdot}{\cdot}_\calX\)~\citep{ryu22book-monotone}, and thus \(-\delta_\theta\) is 1-cocoercive: \(\forall z, w \in \calX\),
\begin{align}
    \label{app:fom:admm-drs}
    \inprod{r^\parallel_\theta(z) + r^\perp_\theta(z) - r^\parallel_\theta(w) - r^\perp_\theta(w)}{z - w}_\calX
    \ge \norm{
        r^\parallel_\theta(z) + r^\perp_\theta(z) - r^\parallel_\theta(w) - r^\perp_\theta(w)
    }_\calX^2. 
\end{align}
For \(\tau \ne 1\), fix \(z, w \in \calX\) and abbreviate \(u := r^\parallel_\theta(z) - r^\parallel_\theta(w) \in \ran(\calA^*)\), \(v := r^\perp_\theta(z) - r^\perp_\theta(w) \in \ker(\calA)\), so that \(\delta_\theta(z) - \delta_\theta(w) = -(u + \tau v)\). Since \(G_\theta\) acts as the identity on \(\ran(\calA^*)\) and as \(\tau^{-1}\Id\) on \(\ker(\calA)\), we have the key identity \(G_\theta (u + \tau v) = u + v\). Therefore, 
\begin{align*}
    & -\inprod{\delta_\theta(z) - \delta_\theta(w)}{z - w}_\theta = \inprod{G_\theta(u + \tau v)}{z - w}_\calX = \inprod{u + v}{z-w}_\calX \\
    \ge & \norm{u+v}_\calX^2 = \norm{G_\theta (\delta_\theta(z) - \delta_\theta (w))}_\calX^2 \\
    \ge & \min\{1,\tau^{-1} \} \inprod{G_\theta (\delta_\theta(z) - \delta_\theta (w))}{\delta_\theta(z) - \delta_\theta (w)}_\calX = \min\{1,\tau^{-1} \} \norm{\delta_\theta(z) - \delta_\theta (w)}_\theta^2,
\end{align*}
where the first inequality is~\eqref{app:fom:admm-drs} and the second holds because \(G_\theta^2 \succeq_\calX \min\{1,\tau^{-1}\} \cdot G_\theta\), which follows by comparing \(\norm{G_\theta w}_\calX^2 = \norm{\PA w}_\calX^2 + \tau^{-2}\norm{\PAp w}_\calX^2\) with \(\inprod{G_\theta w}{w}_\calX = \norm{\PA w}_\calX^2 + \tau^{-1}\norm{\PAp w}_\calX^2\). Thus, \(-\delta_\theta\) is \(\min\{1,\tau^{-1} \}\)-cocoercive in \(\inprod{\cdot}{\cdot}_\theta\). Equivalently, \(T_\theta\) is \(\alpha_\theta\)-averaged with \(\alpha_\theta = 0.5 \max\{1,\tau\}\); note that \(\alpha_\theta \in (0,1)\) precisely because \(\tau \in (0,2)\), which is exactly the admissible multiplier range. 

We now verify Assumption~\ref{ass:asr:T} (\romannumeral4). Since \(U\) is compact, \(M_U^\calA := \sup_{\theta \in U} \normop{\calA}\), \(M_U^b := \sup_{\theta \in U} \norm{b}_\calY\) and \(M_U^c := \sup_{\theta \in U} \norm{c}_\calX\) are finite. Write \(\theta = (\calA, b, c)\), \(\theta' = (\calA', b', c')\) in \(U\), and note \(\normop{\calA - \calA'} \le \normHS{\calA - \calA'} \le \normD{\theta - \theta'}\).

The condition \(\calA \calA^* \succeq_\calY \kappa_U \cdot \Id\) makes \(\calA\calA^*\) uniformly invertible on \(U\), so the resolvent identity \(X^{-1} - Y^{-1} = X^{-1}(Y-X)Y^{-1}\) gives \(\normop{(\calA\calA^*)^{-1} - ({\calA'}{\calA'}^*)^{-1}} \le \kappa_U^{-2} \normop{\calA\calA^* - \calA'{\calA'}^*} \le 2 M_U^\calA \kappa_U^{-2} \normop{\calA - \calA'}\). Since \(\calA^\dagger = \calA^*(\calA\calA^*)^{-1}\) and \(\PA = \calA^\dagger \calA\) are products of maps that are bounded and Lipschitz in \(\calA\) over \(U\), there is a finite \(c_U\), depending only on \(\kappa_U\) and \(M_U^\calA\), with
\begin{align}
    \label{eq:app:fom:admm-proj-lip}
    \normop{\calA^\dagger - {\calA'}^\dagger} \le c_U \normop{\calA - \calA'}, \qquad
    \normop{\PA - P_{\calA'}} = \normop{\PAp - P_{\calA'}^\perp} \le c_U \normop{\calA - \calA'};
\end{align}
moreover \(\normop{\calA^\dagger} \le \kappa_U^{-1/2}\) and \(\normop{\PA}, \normop{\PAp} \le 1\) on \(U\).

Expanding~\eqref{eq:app:fom:admm-Ttheta} as \(T_\theta(z) = z - \PA \Pi_\calK(z) + \PA \calA^\dagger b + \tau \PAp \Pi_{\calK^*}(-z) - \tau\sigma \PAp c\), the identity terms cancel in the difference:
\begin{align*}
    T_\theta(z) - T_{\theta'}(z) = & -(\PA - P_{\calA'}) \Pi_\calK(z) + \tau (\PAp - P_{\calA'}^\perp) \Pi_{\calK^*}(-z) \\
    & + (\PA \calA^\dagger b - P_{\calA'} {\calA'}^\dagger b') - \tau\sigma(\PAp c - P_{\calA'}^\perp c').
\end{align*}
Because \(\calK, \calK^*\) are closed convex cones, \(\Pi_\calK(0) = \Pi_{\calK^*}(0) = 0\), so \(1\)-Lipschitzness gives \(\norm{\Pi_\calK(z)}_\calX \le \norm{z}_\calX\) and \(\norm{\Pi_{\calK^*}(-z)}_\calX \le \norm{z}_\calX\). With~\eqref{eq:app:fom:admm-proj-lip}, the first line is thus at most \((1+\tau) c_U \norm{z}_\calX \normD{\theta - \theta'}\). The second line does not involve \(z\); splitting each product as \(\normop{\PA - P_{\calA'}} \norm{\calA^\dagger b}_\calX + \normop{P_{\calA'}} \norm{\calA^\dagger b - {\calA'}^\dagger b'}_\calX\) and using \(\norm{\calA^\dagger b}_\calX \le \kappa_U^{-1/2} M_U^b\), it is at most \(\kappa_U' \normD{\theta - \theta'}\) for a finite \(\kappa_U'\) depending only on \(\kappa_U, M_U^\calA, M_U^b, M_U^c, \sigma, \tau\). Hence~\eqref{eq:asr:ass-lip} holds with a modulus of the form
\begin{align}
    \label{eq:app:fom:admm-LU}
    L_U(r) = \kappa_U' + (1+\tau) c_U \cdot r,
\end{align}
which lies in \(\Affp\).

The linear growth in \(\norm{z}_\calX\) cannot be removed. Suppose \(\calK\) spans \(\calX\), as it does for \(\R^n_+\), \(\calQ^n\) and \(\psd{n}\), and take \(b = b'\), \(c = c'\) and \(z = td\) with \(d \in \calK\), \(t > 0\). Then \(\Pi_\calK(z) = td\), \(\Pi_{\calK^*}(-z) = 0\), so
\begin{align}
    \label{eq:app:fom:admm-growth}
    T_\theta(z) - T_{\theta'}(z) = -t(\PA - P_{\calA'})d + w, \qquad
    w := \PA\calA^\dagger b - P_{\calA'}{\calA'}^\dagger b + \tau\sigma(\PA - P_{\calA'})c ,
\end{align}
where \(w\) does \emph{not} vanish merely because \(b = b'\) and \(c = c'\), since \(\calA^\dagger \ne {\calA'}^\dagger\) in general; it is, however, independent of \(t\). If \(\PA \ne P_{\calA'}\) then, \(\calK\) spanning \(\calX\), some \(d \in \calK\) has \((\PA - P_{\calA'})d \ne 0\), and \(\norm{T_\theta(z) - T_{\theta'}(z)}_\calX \ge t\norm{(\PA - P_{\calA'})d}_\calX - \norm{w}_\calX \to \infty\). A uniform Lipschitz constant on \(U \times \calX\) is therefore impossible.

\paragraph{sGS-ADMM case.} Fix \(\sigma > 0\) and \(\tau \in (0, 2)\). Suppose \(\calA \calA^* \succeq_\calY \kappa_U \cdot \Id\) with constant \(\kappa_U > 0\). Consider sGS-ADMM for conic-LP~\citep{chen17mp-sgsadmm}: 
\begin{subequations}
    \label{eq:app:fom:sgs}
    \begin{align}
        \ykhalf & = (\calA \calA^*)^{-1} [\sigma^{-1} b - \calA (\sigma^{-1} \xk + \sk - c)], \label{eq:app:fom:sgs-yhalf-update} \\
        \skpo & = \Pi_{\calK^*}(c - \calA^* \ykhalf - \sigma^{-1} \xk), \label{eq:app:fom:sgs-s-update} \\
        \ykpo & = (\calA \calA^*)^{-1} [\sigma^{-1} b - \calA (\sigma^{-1} \xk + \skpo - c)], \label{eq:app:fom:sgs-y-update} \\
        \xkpo & = \xk +\tau \sigma (\calA^* \ykpo + \skpo - c).  \label{eq:app:fom:sgs-x-update}
    \end{align}    
\end{subequations}
Compared to ADMM, sGS-ADMM adds one additional \(y\)-update in~\eqref{eq:app:fom:sgs-y-update}.
Similar to the ADMM case, define the auxiliary state variable \(\zk := \PAp \xk + \calA^\dagger b - \sigma \PA \sk - \sigma \PAp c\). From~\eqref{eq:app:fom:sgs-yhalf-update} and~\eqref{eq:app:fom:sgs-s-update}, \(\skpo = \sigma^{-1} \Pi_{\calK^*}(-\zk)\). From~\eqref{eq:app:fom:sgs-x-update}, we get \(\xkpo = (\Id - \tau \PA) \xk + \tau \sigma \PAp \skpo + \tau \calA^\dagger b - \tau \sigma \PAp c\). With the observation that \(\PAp \xk = \PAp \zk + \sigma \PAp c\), we calculate \(\zkpo\):
\begin{align}
    & \zkpo = \PAp \xkpo + \calA^\dagger b - \sigma \PA \skpo - \sigma \PAp c \nonumber \\
    = & \PAp \xk + \sigma (\tau \PAp - \PA) \skpo + \calA^\dagger b - (1+\tau) \sigma \PAp c \nonumber \\
    = & \PAp \zk + (\tau \PAp - \PA) \Pi_{\calK^*}(-\zk) + \calA^\dagger b - \sigma \tau \PAp c \nonumber \\
    = & \zk - \PA [\Pi_\calK(\zk) - \Pi_{\calK^*}(-\zk)] - \PA \Pi_{\calK^*}(-\zk) + \calA^\dagger b + \tau \PA^\perp \Pi_{\calK^*}(-\zk) - \sigma \tau \PAp c \nonumber \\
    = & \zk - \PA [\Pi_{\calK}(\zk) - \calA^\dagger b] - \tau \PAp [-\Pi_{\calK^*}(-\zk) + \sigma c] \label{eq:app:fom:sgs-Ttheta}.
\end{align}
The form of~\eqref{eq:app:fom:sgs-Ttheta} is exactly the same as that of~\eqref{eq:app:fom:admm-Ttheta}. Therefore, in its \(z\)-variable formula, sGS-ADMM inherits all dynamical properties from ADMM, including \(T_\theta, m_U, M_U, \alpha_\theta, G_\theta\), and in particular Assumption~\ref{ass:asr:T} (\romannumeral4) with the same modulus~\eqref{eq:app:fom:admm-LU}. 

What differentiates sGS-ADMM from ADMM is its variable recovery strategy from \(z\) to \((x,y,s)\). For sGS-ADMM:
\begin{align}
    \label{eq:app:fom:sgs-recover}
    \xkpo = (\Id - \tau \PA) \xk + \tau \PAp (\Pi_{\calK^*}(-\zk) - \sigma c)  + \tau \calA^\dagger b, \ \skpo = \sigma^{-1} \Pi_{\calK^*}(-\zk), 
\end{align}
and \(\ykpo\) can be recovered from~\eqref{eq:app:fom:sgs-yhalf-update}--\eqref{eq:app:fom:sgs-y-update}. Except for \(\skpo\)'s recovery, these rules are different from ADMM's in~\eqref{eq:app:fom:admm-recover}. Consequently, the observable KKT residuals are also different. 

\paragraph{PDHG case.} Fix \(\eta_x, \eta_y > 0\). We further assume \(\gamma_U := \sup_{\theta \in U} \sqrt{\eta_x \eta_y} \normop{\calA} < 1\). Let \(\calH := \calX \times \calY\) and the auxiliary state variable \(z := (x,y) \in \calH\). PDHG's update for conic-LP is~\citep{chambolle11jmiv-pdhg}:
\begin{subequations}
    \label{eq:app:fom:pdhg}
    \begin{align}
        \xkpo & = x_\theta^+(\zk) := \Pi_\calK[\xk - \eta_x (c - \calA^* \yk)], \label{eq:app:fom:pdhg-x-update} \\
        \ykpo & = y_\theta^+(\zk) := \yk + \eta_y [b - \calA (2 \xkpo - \xk)].  \label{eq:app:fom:pdhg-y-update}
    \end{align}
\end{subequations}
Let \(G_\theta := \begin{bmatrix}
    \eta_x^{-1} \cdot \Id_\calX & \calA^* \\ \calA & \eta_y^{-1} \cdot \Id_\calY
\end{bmatrix}\). Since \(\gamma_U < 1\), \(G_\theta \succ_\calH 0\). Denote \(T_\theta(z) := (x_\theta^+(z), y_\theta^+(z))\). From~\citep{chan14arxiv-inertial-primal-dual-alg}, \(T_\theta\) is firmly nonexpansive (\ie 1/2-averaged) under \(G_\theta\)-metric.

For Assumption~\ref{ass:asr:T} (\romannumeral4), again put \(M_U^\calA := \sup_{\theta \in U} \normop{\calA}\), \(M_U^c := \sup_{\theta \in U} \norm{c}_\calX\), both finite by compactness of \(U\), and recall \(\normop{\calA - \calA'} \le \normD{\theta - \theta'}\). Fix \(z = (x,y) \in \calH\) and \(\theta, \theta' \in U\). Since \(\Pi_\calK\) is \(1\)-Lipschitz, the \(x\)-block obeys
\begin{align}
    \label{eq:app:fom:pdhg-x-param}
    \norm{x_\theta^+(z) - x_{\theta'}^+(z)}_\calX
    \le \eta_x \norm{c - c'}_\calX + \eta_x \normop{\calA - \calA'} \norm{y}_\calY
    \le \eta_x (1 + \norm{z}_\calH) \normD{\theta - \theta'}.
\end{align}
Moreover \(\Pi_\calK(0) = 0\) gives the uniform bound \(\norm{x_{\theta'}^+(z)}_\calX \le \norm{x}_\calX + \eta_x(\norm{c'}_\calX + \normop{\calA'}\norm{y}_\calY) \le (1 + \eta_x M_U^\calA)\norm{z}_\calH + \eta_x M_U^c\), hence \(\norm{2 x_{\theta'}^+(z) - x}_\calX \le (3 + 2\eta_x M_U^\calA) \norm{z}_\calH + 2\eta_x M_U^c\). For the \(y\)-block, inserting \(\calA(2x_{\theta'}^+(z) - x)\) and using~\eqref{eq:app:fom:pdhg-x-param},
\begin{align*}
    & \norm{y_\theta^+(z) - y_{\theta'}^+(z)}_\calY
    \le \eta_y \norm{b - b'}_\calY + 2 \eta_y \normop{\calA} \norm{x_\theta^+(z) - x_{\theta'}^+(z)}_\calX \\
    & \qquad + \eta_y \normop{\calA - \calA'} \norm{2 x_{\theta'}^+(z) - x}_\calX \\
    \le \, & \eta_y \bigl[ 1 + 2 \eta_x M_U^\calA (1 + \norm{z}_\calH) + (3 + 2\eta_x M_U^\calA)\norm{z}_\calH + 2 \eta_x M_U^c \bigr] \cdot \normD{\theta - \theta'}.
\end{align*}
Adding the two blocks yields~\eqref{eq:asr:ass-lip} with the concrete modulus
\begin{align}
    \label{eq:app:fom:pdhg-LU}
    L_U(r) = \bigl[\eta_x + \eta_y(1 + 2\eta_x M_U^\calA + 2 \eta_x M_U^c)\bigr] + \bigl[\eta_x + \eta_y(3 + 4 \eta_x M_U^\calA)\bigr] \cdot r,
\end{align}
again an element of \(\Affp\). As with ADMM, the factor \(\norm{z}_\calH\) is unavoidable, though here the argument needs one more step: the \emph{inputs} to \(\Pi_\calK\) separate by \(\eta_x\norm{(\calA - \calA')^*y}_\calX \to \infty\), but \(\Pi_\calK\) is \(1\)-Lipschitz and could in principle absorb that growth. We now show that it does not. Take \(\calK = \R^n_+\), \(b = b'\), \(c = c'\), \(x = 0\) and \(z = (0, ty)\); then \(x^+ - {x'}^+ = \Pi_{\R^n_+}\bigl[t\eta_x\calA^*y - \eta_x c\bigr] - \Pi_{\R^n_+}\bigl[t\eta_x{\calA'}^*y - \eta_x c\bigr]\), which for large \(t\) behaves like \(t\eta_x\bigl[(\calA^*y)_+ - ({\calA'}^*y)_+\bigr]\). Were \((\calA^*y)_+ = ({\calA'}^*y)_+\) for every \(y \in \calY\), then applying this to both \(y\) and \(-y\) coordinatewise would force \(\calA^*y = {\calA'}^*y\) for all \(y\), \ie \(\calA = \calA'\). Hence whenever \(\calA \ne \calA'\) some \(y\) makes the \(x\)-block grow linearly in \(t\).

It remains to determine \(m_U\) and \(M_U\). Take any \(z = (u,w) \in \calX \times \calY\), \(\inprod{G_\theta z}{z}_\calH = \eta_x^{-1} \norm{u}_\calX^2 + \eta_y^{-1} \norm{w}_\calY^2 + 2\inprod{\calA u}{w}_\calY\). By Cauchy--Schwarz and Young's inequality,
\begin{align*}
    2 \abs{\inprod{\calA u}{w}_\calY} \le \sqrt{\eta_x \eta_y} \normop{\calA} (\eta_x^{-1} \norm{u}_\calX^2 + \eta_y^{-1} \norm{w}_\calY^2) \le \gamma_U (\eta_x^{-1}\norm{u}_\calX^2 + \eta_y^{-1} \norm{w}_\calY^2). 
\end{align*}
Since \(\min\{\eta_x^{-1}, \eta_y^{-1}\} \norm{z}_\calH^2 \le \eta_x^{-1} \norm{u}_\calX^2 + \eta_y^{-1} \norm{w}_\calY^2 \le \max\{\eta_x^{-1}, \eta_y^{-1}\} \norm{z}_\calH^2\), we can take \(m_U := (1 - \gamma_U) / \max\{\eta_x, \eta_y\}\) and \(M_U := (1 + \gamma_U) / \min\{\eta_x, \eta_y\}\).

\begin{proof}[Proof of Theorem~\ref{thm:asr:partial-superposition}]
    Write \(\theta := (\calA, b, c)\), \(\theta' := (\calA, b', c')\). For ADMM/sGS-ADMM, Table~\ref{tab:fom} gives \(\delta_\theta z - \delta_{\theta'} z = \calA^\dagger (b - b') - \tau \sigma \PAp ( c- c')\), which is~\eqref{eq:asr:ass-const-field}. For PDHG, let \(\Delta x^+ := \Pi_\calK(x + \eta_x \calA^* y - \eta_x c) - \Pi_\calK(x + \eta_x \calA^* y - \eta_x c')\), so that \(\delta_\theta z - \delta_{\theta'} z = (\Delta x^+,\ \eta_y (b - b') - 2 \eta_y \calA \Delta x^+)\). Since \(\Pi_\calK\) is 1-Lipschitz, \(\norm{\Delta x^+}_\calX \le \eta_x \norm{c - c'}_\calX\), giving~\eqref{eq:asr:ass-lip-strong}; if \(c = c'\) then \(\Delta x^+ = 0\) and \(\delta_\theta z - \delta_{\theta'} z = (0, \eta_y (b - b'))\), giving~\eqref{eq:asr:ass-const-field}.
\end{proof}

\paragraph{Fixed points correspond to KKT points.} It remains to verify Assumption~\ref{ass:asr:T}~(\romannumeral1).
The equivalence \(\Fix(T_\theta) \ne \emptyset\) iff \(\KKT(\theta) \ne \emptyset\) already follows from the classical fixed-point characterizations, of Douglas--Rachford splitting~\citep[Prop.~26.1]{bauschke17springer-convex-analysis-hilbert-spaces} and of the PDHG saddle point~\citep{chambolle11jmiv-pdhg}. Beyond that, we need the explicit correspondence in the \(z\)-coordinates of Table~\ref{tab:fom}.

\begin{lemma}[Fixed points correspond to KKT points]
    \label{lem:app:fom:fix-kkt}
    Let \(\theta = (\calA,b,c) \in U\). For ADMM and sGS-ADMM under the conditions of Table~\ref{tab:fom}, the map \(\Phi: (x,y,s) \mapsto x - \sigma s\) is a bijection from \(\KKT(\theta)\) onto \(\Fix(T_\theta)\) (the rank condition of Table~\ref{tab:fom} makes \(y\) unique given \((x,s)\)); for PDHG, \((x,y) \in \Fix(T_\theta)\) iff \((x,y,c - \calA^*y) \in \KKT(\theta)\). In both cases \(\Fix(T_\theta) \ne \emptyset\) iff \(\KKT(\theta) \ne \emptyset\).
\end{lemma}
\begin{proof}
    \emph{ADMM and sGS-ADMM.} By Table~\ref{tab:fom} the two share the \(z\)-map, so one computation suffices. Put \(X := \Pi_\calK(z)\) and \(\sigma S := \Pi_{\calK^*}(-z)\); Moreau's decomposition makes \(z \mapsto (X,S)\) a bijection onto the pairs with \(X \in \calK\), \(S \in \calK^*\), \(\inprod{X}{S}_\calX = 0\), with inverse \(z = X - \sigma S\). In these terms Table~\ref{tab:fom} reads \(\delta_\theta(z) = -\PA[X - \calA^\dagger b] - \tau\sigma\PAp[c - S]\), whose two terms lie in the orthogonal subspaces \(\ran(\calA^*)\) and \(\ran(\calA^*)^\perp\); as \(\tau\sigma \ne 0\), \(\delta_\theta(z) = 0\) iff both vanish. The first says \(\calA X = b\), since \(\calA\calA^\dagger = \Id_\calY\) makes it \(\calA^\dagger(\calA X - b) = 0\) with \(\calA^\dagger\) injective; the second says \(c - S \in \ran(\calA^*)\), \ie \(\calA^*y + S = c\) for some \(y\). These are exactly~\eqref{eq:asr:kkt}, and \(y\) is unique because \(\ker(\calA^*) = \{0\}\), which gives the stated bijection.

    \emph{PDHG.} Here \(z = (x,y)\) and \(T_\theta z = (x^+, y + \eta_y[b - \calA(2x^+ - x)])\). If \(z\) is fixed then \(x^+ = x\), and the \(y\)-block gives \(\calA x = b\). With \(s := c - \calA^*y\), the \(x\)-block reads \(x = \Pi_\calK[x - \eta_x s]\), which for a closed convex cone \(\calK\) and \(\eta_x > 0\) holds iff \(-\eta_x s \in N_\calK(x)\), \ie iff \(x \in \calK\), \(s \in \calK^*\) and \(\inprod{x}{s}_\calX = 0\). Thus \((x,y,s) \in \KKT(\theta)\). The converse follows by reversing each step.
\end{proof}

\paragraph{From the \(z\)-residual to KKT residuals.}
The theory measures \(\norm{\delta_\theta(z)}_\calH\), while solvers monitor the KKT residuals of \((x,y,s)\). The next lemma connects the two for the maps of Table~\ref{tab:fom}. 

\begin{lemma}[KKT-residual bridge]
    \label{lem:app:fom:kkt-bridge}
    Let \(\theta = (\calA,b,c) \in U\). For ADMM/sGS-ADMM under the conditions of Table~\ref{tab:fom}, take any \(z \in \calX\), put \(X := \Pi_\calK(z)\) and \(\sigma S := \Pi_{\calK^*}(-z)\), so that \(X \in \calK\), \(S \in \calK^*\), \(\inprod{X}{S}_\calX = 0\), and define the projected KKT residual \(r_\theta(z) := \norm{\calA X - b}_\calY + \min_{y \in \calY} \norm{\calA^* y + S - c}_\calX\). Then
    \begin{align}
        \label{eq:app:fom:kkt-bridge}
        \tfrac{1}{2} \min\{1/M_U^\calA,\ \tau\sigma\} \cdot r_\theta(z) \ \le\ \norm{\delta_\theta(z)}_\calX \ \le\ \max\{\kappa_U^{-1/2},\ \tau\sigma\} \cdot r_\theta(z).
    \end{align}
    For PDHG, write \(z = (x,y)\) and \(\delta_\theta(z) = (\delta_x, \delta_y)\), and put \(x^+ := x + \delta_x\) and \(\tilde{s} := c - \calA^* y + \eta_x^{-1} \delta_x\), so that \(x^+ \in \calK\), \(\tilde{s} \in \calK^*\), \(\inprod{x^+}{\tilde{s}}_\calX = 0\). Then, with \(r_\theta(z) := \norm{\calA x^+ - b}_\calY + \norm{\calA^* y + \tilde{s} - c}_\calX\),
    \begin{align}
        \label{eq:app:fom:kkt-bridge-pdhg}
        (\eta_x^{-1} + \eta_y^{-1} + M_U^\calA)^{-1} \cdot r_\theta(z) \ \le\ \norm{\delta_\theta(z)}_\calH \ \le\ (\eta_x + \eta_y + \eta_x \eta_y M_U^\calA) \cdot r_\theta(z).
    \end{align}
\end{lemma}
\begin{proof}
    For ADMM/sGS-ADMM, the proof of Lemma~\ref{lem:app:fom:fix-kkt} gives \(\delta_\theta(z) = -\PA[X - \calA^\dagger b] - \tau\sigma \PAp[c - S]\), with \(\PA[X - \calA^\dagger b] = \calA^\dagger(\calA X - b)\) since \(\PA \calA^\dagger = \calA^\dagger\), and the two terms are orthogonal. Since \(\calA \calA^\dagger = \Id_\calY\) and \(\normop{\calA^\dagger} \le \kappa_U^{-1/2}\), we get \(\norm{\calA X - b}_\calY / M_U^\calA \le \norm{\calA^\dagger (\calA X - b)}_\calX \le \kappa_U^{-1/2} \norm{\calA X - b}_\calY\), while \(\norm{\PAp(c - S)}_\calX\) is the distance from \(c - S\) to \(\ran(\calA^*)\), namely \(\min_{y \in \calY} \norm{\calA^* y + S - c}_\calX\). The upper bound in~\eqref{eq:app:fom:kkt-bridge} is the triangle inequality. For the lower bound, orthogonality gives \(\norm{\delta_\theta(z)}_\calX \ge \max\{\norm{\calA^\dagger(\calA X - b)}_\calX,\ \tau\sigma \norm{\PAp(c-S)}_\calX\}\), and the maximum of two nonnegative numbers is at least half their sum. For PDHG, \(x^+ = \Pi_\calK(w)\) with \(w := x - \eta_x(c - \calA^* y)\), and Moreau's decomposition of \(w\) gives \(x^+ \in \calK\), \(w - x^+ \in -\calK^*\) and \(\inprod{x^+}{w - x^+}_\calX = 0\); since \(w - x^+ = -\eta_x \tilde{s}\), the pair \((x^+, \tilde{s})\) is exactly cone-feasible and complementary. Its dual residual is \(\calA^* y + \tilde{s} - c = \eta_x^{-1} \delta_x\), while the \(y\)-block \(\delta_y = \eta_y [b - \calA(2x^+ - x)]\) gives \(b - \calA x^+ = \eta_y^{-1} \delta_y + \calA \delta_x\). Both sides of~\eqref{eq:app:fom:kkt-bridge-pdhg} follow from the triangle inequality applied to these two identities, using \(\norm{\delta_x}_\calX, \norm{\delta_y}_\calY \le \norm{\delta_\theta(z)}_\calH \le \norm{\delta_x}_\calX + \norm{\delta_y}_\calY\).
\end{proof}
In words, for ADMM/sGS-ADMM the \(z\)-residual \emph{is} the projected KKT residual of the recovered pair \((X,S)\), up to constants depending only on \(U\), \(\sigma\), and \(\tau\): the pair is exactly cone-feasible and complementary, its primal infeasibility is the \(\ran(\calA^*)\)-component of \(\delta_\theta(z)\), and its minimal dual infeasibility is the orthogonal component. For PDHG the same reading holds at the half-updated pair \((x^+, \tilde{s})\), which one projection computes from the state. 

\subsection{Additional material for \S\ref{sec:asr:asr}}
\label{app:asr}

\begin{lemma}[Fixed point sets are closed and convex]
    \label{lem:app:asr:fix-cc}
    Under Assumption~\ref{ass:asr:T}, \(\Fix(T_\theta)\) is closed and convex for every \(\theta \in U\).
\end{lemma}
\begin{proof}
    \(T_\theta\) is nonexpansive in the Hilbertian metric \(\inprod{\cdot}{\cdot}_\theta\), and the fixed point set of a nonexpansive map on a Hilbert space is closed and convex~\citep[Corollary~4.24]{bauschke17springer-convex-analysis-hilbert-spaces}.
\end{proof}

\begin{proof}[Proof of Proposition~\ref{prop:asr:sr}]
    Fix \(\eps > 0\) and \(z \in \calS_\eps\). Since \(\TEpsC\) is \(\alpha_{\thetaEpsC}\)-averaged with respect to \(\inprod{\cdot}{\cdot}_{\thetaEpsC}\), it is nonexpansive in that dynamic metric. Thus, \(\forall k \ge 1\), \(\norm{(\TEpsC)^{k+1}(z) - (\TEpsC)^{k}(z)}_{\thetaEpsC} \le \norm{\TEpsC(z) - z}_{\thetaEpsC} = \norm{\deltaEpsC(z)}_{\thetaEpsC}\). On the other hand, \(\abs{\dist_{\thetaEpsC}((\TEpsC)^k (z), \VEpsC) - \dist_{\thetaEpsC}(z, \VEpsC)} \le \norm{(\TEpsC)^k (z) - z}_{\thetaEpsC}\) by \(1\)-Lipschitzness of the distance function. Thus, \(\forall k \ge 1\),
    \begin{align*}
        & \dist_{\thetaEpsC}((\TEpsC)^k (z), \VEpsC) \ge \dist_{\thetaEpsC}(z, \VEpsC) - \norm{(\TEpsC)^k (z) - z}_{\thetaEpsC} \\
        \ge & \dist_{\thetaEpsC}(z, \VEpsC) - \sum_{i=0}^{k-1} \norm{(\TEpsC)^{i+1} (z) - (\TEpsC)^{i} (z)}_{\thetaEpsC} \ge \dist_{\thetaEpsC}(z, \VEpsC) - k \norm{\deltaEpsC(z)}_{\thetaEpsC}.
    \end{align*}
    From Assumption~\ref{ass:asr:T}, for any \(w \in \calH\), we get \(\sqrt{m_U} \cdot \dist_{\calH}(w, \VEpsC) \le \dist_{\thetaEpsC}(w, \VEpsC) \le \sqrt{M_U} \cdot \dist_{\calH}(w, \VEpsC)\), \(\norm{w}_\thetaEpsC \le \sqrt{M_U} \cdot \norm{w}_\calH\). Thus, \(\forall k \ge 1\), \(\sqrt{M_U} \cdot \dist_\calH((\TEpsC)^k(z), \VEpsC) \ge \sqrt{m_U} \cdot \dist_{\calH}(z, \VEpsC) - k \sqrt{M_U} \cdot \norm{\deltaEpsC(z)}_{\calH}\). From~\eqref{eq:asr:sr-def} and \(k \le N_\eps\), we have 
    \begin{align*}
        \dist_\calH((\TEpsC)^k(z), \VEpsC) \ge (\sqrt{m_U / M_U} - k \rho_\eps) \cdot \dist_{\calH}(z, \VEpsC) \ge \sqrt{m_U/M_U} \beta \cdot \dist_\calH(z, \VEpsC)
    \end{align*}
    for any \(z \in \calS_\eps\).
\end{proof}

\begin{proof}[Proof of Proposition~\ref{prop:asr:finite-union}]
    Fix \(\eps > 0\). The union is nonempty and misses \(\VEpsC\), since each \(\calS^i_\eps\) does, which is Definition~\ref{def:asr:sr-def}~(\romannumeral1). Any \(z \in \cup_{i \in [r]} \calS_\eps^i\) lies in some \(\calS^i_\eps\), so~\eqref{eq:asr:sr-def} gives \(\norm{\deltaEpsC(z)}_\calH \le \rho^i_\eps \cdot \dist_\calH(z, \VEpsC) \le \max_{i \in [r]}\rho_\eps^i \cdot \dist_\calH(z, \VEpsC)\). Finiteness of the index set gives \(\max_{i \in [r]}\rho_\eps^i \to 0\) as \(\eps \downarrow 0\). Conversely each \(\calS^i_\eps \subseteq \cup_{j \in [r]}\calS^j_\eps\), so the supremum in~\eqref{eq:asr:sr-def} over the union is at least \(\rho^i_\eps\) for every \(i\), and the modulus is exactly \(\max_{i \in [r]}\rho^i_\eps\).
\end{proof}

Proposition~\ref{prop:asr:finite-union} does not extend verbatim to countably many families: \(\rho^i_\eps \to 0\) for each fixed \(i\) does not force \(\sup_i \rho^i_\eps \to 0\), as \(\rho^i_\eps := \eps^{1/i}\) shows.

\begin{proof}[Proof of Proposition~\ref{prop:asr:thickening}]
    Since \(\TEpsC\) is nonexpansive in its dynamic metric, for any \(z, q \in \calH\), we have 
    \begin{align*}
        & \norm{\deltaEpsC(z) - \deltaEpsC(q)}_\calH 
        \le 1/\sqrt{m_U} \cdot (\norm{\TEpsC(z) - \TEpsC(q)}_\thetaEpsC + \norm{z - q}_\thetaEpsC) \\
        \le & 2/\sqrt{m_U} \cdot \norm{z - q}_\thetaEpsC \le 2\sqrt{M_U/m_U} \cdot \norm{z - q}_\calH.
    \end{align*}
    Now take \(z \in \tilde{\calS}_\eps\) and \(q \in \calS_\eps\) with \(\norm{z - q}_\calH \le a_\eps \cdot \dist_\calH(q, \VEpsC)\). By \(1\)-Lipschitzness of the distance function, \(\dist_\calH(z, \VEpsC) \ge \dist_\calH(q, \VEpsC) - \norm{z-q}_\calH \ge (1-a_\eps)\cdot \dist_\calH(q, \VEpsC) > 0\), since \(a_\eps < 1\) and \(q \notin \VEpsC\); as \(\calS_\eps \subseteq \tilde\calS_\eps\), this settles Definition~\ref{def:asr:sr-def}~(\romannumeral1). Finally,
    \begin{align*}
        & \norm{\deltaEpsC(z)}_\calH \le \norm{\deltaEpsC(q)}_\calH + 2\sqrt{M_U/m_U} \cdot \norm{z - q}_\calH \\
        \le &  (\rho_\eps + 2\sqrt{M_U/m_U} \cdot a_\eps) \cdot \dist_\calH(q, \VEpsC)
        \le \frac{\rho_\eps + 2\sqrt{M_U/m_U} \cdot a_\eps}{1 - a_\eps} \cdot \dist_\calH(z, \VEpsC).
    \end{align*}
    The last inequality uses the displayed lower bound on \(\dist_\calH(z,\VEpsC)\), and \(\tilde\rho_\eps \to 0\) because \(\rho_\eps, a_\eps \to 0\).
\end{proof}

\begin{proof}[Proof of Proposition~\ref{prop:asr:forward-drift}]
    The first statement is classical, and needs more than nonexpansiveness. Since \(T_\theta\) is nonexpansive in \(\inprod{\cdot}{\cdot}_\theta\), the set \(\cl{\ran(\Id - T_\theta)}\) is convex~\citep{pazy71ijm-asymptotic-behavior-contractions}, hence so is \(\cl{\ran(\delta_\theta)} = -\cl{\ran(\Id - T_\theta)}\), so the projection defining \(v_\theta\) is single-valued and \(-v_\theta\) is the \emph{minimal displacement vector} of \(T_\theta\). Being \(\alpha_\theta\)-averaged, \(T_\theta\) is in addition strongly nonexpansive in that metric~\citep{bauschke17springer-convex-analysis-hilbert-spaces}, and for strongly nonexpansive maps the successive differences \((\Id - T_\theta)\zk = -\delta_\theta \zk\) converge strongly to \(-v_\theta\), from any \(z^{(0)}\)~\citep{bruck77hjm-nonexpansive-projection,baillon78hjm-asymptotic-behavior-nonexpansive-mappings}; plain nonexpansiveness yields only convergence of their norms. Monotonicity is immediate from nonexpansiveness of \(T_\theta\) in \(\norm{\cdot}_\theta\): since \(\delta_\theta \zkpo = T_\theta(T_\theta \zk) - T_\theta \zk\), we get \(\norm{\delta_\theta \zkpo}_\theta \le \norm{T_\theta \zk - \zk}_\theta = \norm{\delta_\theta \zk}_\theta\); the limit is \(\norm{v_\theta}_\theta\) by the first statement. Note the ambient residual \(\norm{\delta_\theta \zk}_\calH\) need not be monotone, which is why the dynamic metric is the natural one to monitor. For the second statement, the projection definition gives \(\inprod{v_\theta}{v_\theta - w}_\theta \le 0, \forall w \in \cl{\ran(\delta_\theta)}\). On the other hand, since \(T_\theta\) is \(\alpha_\theta\)-averaged under \(\inprod{\cdot}{\cdot}_\theta\), \(-\delta_\theta\) is \(1/(2\alpha_\theta)\)-cocoercive: 
    \(
        \tfrac{1}{2\alpha_\theta} \norm{\delta_\theta(z) - \delta_\theta(z+t v_\theta)}_\theta^2 \le \inprod{\delta_\theta(z) - \delta_\theta(z+t v_\theta)}{t v_\theta}_\theta
    \). With \(\delta_\theta(z) = v_\theta\) the right side is \(t \inprod{v_\theta - \delta_\theta(z + t v_\theta)}{v_\theta}_\theta \le 0\), by the projection inequality applied to \(w := \delta_\theta(z + t v_\theta) \in \ran(\delta_\theta)\); hence \(\norm{v_\theta - \delta_\theta(z+t v_\theta)}_\theta = 0\). For the last statement, note \(\Fix(T_\theta) = \emptyset\) means \(0 \notin \ran(\delta_\theta)\). If \(v_\theta \in \ran(\delta_\theta)\), then \(v_\theta \ne 0\); picking \(z\) with \(\delta_\theta(z) = v_\theta\) and setting \(z_j := z + j v_\theta\), the previous statement gives \(\delta_\theta(z_j) = v_\theta\) while \(\norm{z_j}_\calH \to \infty\). Otherwise \(v_\theta \in \cl{\ran(\delta_\theta)} \setminus \ran(\delta_\theta)\), so there are \(w_j\) with \(\delta_\theta(w_j) \to v_\theta\); such \(\{w_j\}\) must be unbounded, since a bounded one would admit, by finite dimensionality and continuity of \(\delta_\theta\), a subsequence converging to some \(w\) with \(\delta_\theta(w) = v_\theta\). Extracting a subsequence with norms tending to \(\infty\) finishes the proof.
\end{proof}

\subsection{Additional material for \S\ref{sec:exp}}
\label{app:exp}

\subsubsection{Motivating example}
\label{app:exp:motivating}
Here \(\barcalA = [1\ \ -1]\), so \(\barcalA\barcalA^* = 2\) and \(\normop{\barcalA} = \sqrt2\). Table~\ref{tab:fom} asks for \(\calA\calA^* \succeq_\calY \kappa_U \cdot \Id\) for ADMM/sGS-ADMM and \(\gamma_U < 1\) for PDHG; at \(\bartheta\) these read \(2 > 0\) and \(\sqrt{2\eta_x\eta_y} < 1\). Both are strict inequalities between continuous functions of \(\theta\), so both persist on a small compact \(U \ni \bartheta\), and Assumption~\ref{ass:asr:T} holds there. Since \(\thetaEpsC\) and \(\bartheta\) differ only in \(b\), \(U\) may be taken with \(\calA\) fixed, and Theorem~\ref{thm:asr:partial-superposition}\,(\romannumeral1) makes the defect state-free: \(\deltaEpsC - \deltaEpsP \equiv \barcalA^\dagger\eps = (\eps/2,\,-\eps/2)\) for ADMM and \((0,0,\eta_y\eps)\) for PDHG, of size \(e_\eps = \eps/\sqrt2\), resp.\ \(\eta_y\eps\), in \(\norm{\cdot}_\calH\). As \(\thetaEpsP \equiv \bartheta\), \(\deltaEpsP\) vanishes on \(\calL\), so~\eqref{eq:asr:asr-def} holds with \(\mu_\eps \equiv 0\) and \eqref{eq:asr:asr-strength} needs only \(e_\eps\) divided by \(\inf_{z \in \calS}\dist_\calH(z,\VEpsC)\). That infimum is at least \(c_0\): for ADMM, \(z = (-(1-t)\sigma,-t\sigma)\) gives \(\norm{z - (\eps,-\sigma)}_\calH^2 = ((1-t)\sigma+\eps)^2 + ((1-t)\sigma)^2 \ge 2((1-t)\sigma)^2 = \dist_\calH(z,F)^2\), and for PDHG, \(z = (0,0,y)\) gives \(\norm{z - (\eps,0,1)}_\calH^2 = \eps^2 + (1-y)^2 \ge \dist_\calH(z,F)^2\), while \(\dist_\calH(z,F) \ge c_0\) on \(\calS\) by construction. Hence \(\rho_\eps {}\le \eps/(\sqrt2c_0)\), resp.\ \(\eta_y\eps/c_0\). Replacing the fixed margin by \(c_\eps := c_\star\eps^\alpha\) with \(c_\star > 0\) and \(\alpha \ge 0\) leaves every step above unchanged, \(\calS_\eps\) staying nonempty once \(c_\eps\) falls below the length of \(\calL\), and gives \(\rho_\eps = \Theta(\eps^{1-\alpha})\) while \(c_\eps \gg \eps\), that is for \(\alpha < 1\); for \(\alpha \ge 1\) the \(O(\eps)\) displacement of \(\VEpsC\) dominates the margin, the infimum saturates at \(\Theta(\eps)\) and \(\rho_\eps = \Theta(1)\). Either way \(\rho_\eps \to 0\) exactly when \(\alpha < 1\); Definition~\ref{def:asr:asr} constrains only \(\rho_\eps\), and not how \(\calS_\eps\) moves with \(\eps\).

\emph{Which \(\eps\)-independent states are slow.} For a fixed singleton \(\calS_\eps \equiv \{z\}\), \eqref{eq:asr:sr-def} reads \(\rho_\eps = \norm{\deltaEpsC(z)}_\calH/\dist_\calH(z,\VEpsC) =: R_\eps(z)\). Only \(b\) moves, so \(\deltaEpsC = \delta_\bartheta + \barcalA^\dagger\eps\) everywhere and \(\VEpsC = \{F_\eps\}\) with \(F_\eps \to F\); hence \(R_\eps(z) \to \norm{\delta_\bartheta(z)}_\calH/\norm{z-F}_\calH\) for every fixed \(z \ne F\), which vanishes exactly when \(\delta_\bartheta(z) = 0\):
\begin{align*}
    \{z\} \text{ is an SR family} \iff z \in \calL \setminus \{F\}.
\end{align*}
At \(F\) the two scales coincide, \(\eps/\sqrt2\) against \(\eps\), so \(R_\eps(F) \equiv 2^{-1/2}\) (\(\eta_y\) for PDHG). The conclusion is uniform off \(\calL\): on a compact \(K \subset \calH \setminus \calL\), \(m := \inf_K\norm{\delta_\bartheta}_\calH > 0\) and \(\sup_K\dist_\calH(\cdot,\VEpsC) \le M < \infty\), while~\eqref{eq:asr:ass-lip-param} makes \(\deltaEpsC \to \delta_\bartheta\) uniformly on \(K\), so \(\inf_K R_\eps \ge m/(2M) > 0\) for small \(\eps\). This bounds the local step, not the whole run, since an orbit may enter the afterimage later; on one ray, however, it also bounds the whole run. For ADMM with \(\sigma = \tau = 1\) and \(t \le 0\), both coordinates of \(z = (t,t-1)\) are nonpositive, so \(\Pi_\calK(z) = 0\) and Table~\ref{tab:fom} sends the whole ray to \(T_{\thetaEpsC}(t,t-1) = (\eps/2,\,-1-\eps/2)\), at distance \(\eps/\sqrt2\) from \(F_\eps\). Writing \(z = F_\eps + w\) on the quadrant \(z_1 \ge 0 \ge z_2\) makes the map affine with linear part \(0.5\bigl[\begin{smallmatrix}1&-1\\1&1\end{smallmatrix}\bigr] = 2^{-1/2}Q\), \(Q\) orthogonal; the orbit stays there as long as \(\norm{w}_\calH < \eps\), which holds at \(k = 1\) and persists under the contraction; so \(\dist_\calH(\zk,\VEpsC) = \eps\,2^{-k/2}\) exactly for every \(k \ge 1\), and one step reaches any fixed accuracy \(\zeta > 0\) once \(\eps \le \sqrt2\,\zeta\).

\emph{Restart and Halpern acceleration do not remove this plateau.}
Restart and Halpern-type schemes are popular accelerations for FOM on LP~\citep{applegate23mp-faster,lu24-restarted}, so it is natural to ask whether they shorten the plateau above. For the fixed maps of Table~\ref{tab:fom}, they do not. On \(\calL\), \(\TEpsC z = z + d_\eps\) with \(d_\eps := \barcalA^\dagger \eps = (\eps/2,\, -\eps/2)\) parallel to \(\calL\), so the plain iteration advances by \(\norm{d_\eps}_\calH = \eps/\sqrt2\) per step and needs \(\Theta(\eps^{-1})\) steps to traverse an \(O(1)\) portion of \(\calL\). Here \(\Theta(\eps^{-1})\) counts operator evaluations for a fixed relative reduction of \(\dist_\calH(z, \VEpsC)\); it is not the time to reach a fixed residual tolerance, which the \(O(\eps)\) residual on \(\calL\) meets immediately. Restarts are counted by their operator evaluations. For the Halpern iteration \(\zkpo = \tfrac{1}{k+2} z^{(0)} + \tfrac{k+1}{k+2} \TEpsC \zk\), induction gives \(\zk = z^{(0)} + \tfrac{k}{2} d_\eps\) while the orbit stays on \(\calL\) --- half the speed of the plain iteration; applied to the reflected map \(2\TEpsC - \Id\), the same computation gives \(\zk = z^{(0)} + k d_\eps\), recovering the plain speed. Restarting either scheme only resets the anchor \(z^{(0)}\) and leaves the displacement per operator evaluation at \(O(\eps)\). Hence these accelerations change constants, but not the \(\Theta(\eps^{-1})\) plateau; the same computation applies to PDHG with \(d_\eps = (0,0,\eta_y\eps)\). This does not contradict the restart guarantees of~\citet{applegate23mp-faster}, which accelerate with respect to a fixed instance's sharpness constant: along our family that constant degenerates, as the error-bound reading of \(\rho_\eps\) in \S\ref{sec:asr:asr} records. Solvers that adapt \(\sigma\), the step sizes, or the scaling fall outside this computation. The constant-step phase itself is a known regime of ADMM on LP/QP~\citep{boley13siopt-linearconv-admm-lp-qp}; the afterimage certificate explains it through a neighboring problem and extends it to non-polyhedral and non-attained limits.

\subsubsection{SDP square-root H\"older}
\label{app:exp:sqrt}
The \(\Theta(h_\eps^{1/2})\) dislocation below is consistent with the square-root error bounds for LMIs~\citep{sturm00siopt-error-bound-lmi,ding23ol-sc-errorbound-sensitivity}; we derive it directly, since this center has no strictly feasible dual slack and strict complementarity fails at the collapse endpoint.
Throughout \(\sigma = \tau = 1\), \(\calK = \psd{2}\) is self-dual and \(\barcalA\barcalA^* = \normF{A_1}^2 = 3\), so \(\calA^\dagger b = \tfrac{b}{3}A_1\) and \(\PAp Y = Y - \tfrac13\inprod{A_1}{Y}A_1\). Both fixed sets are the images \(Z = X - \sigma S\) of the KKT triples \((X,y,S)\).

At the center, the primal minimizes \(X_{22}\) subject to \(2X_{12} = X_{22}\), \(X \succeq 0\); positivity gives \(X_{22} \ge 0\) while \(X = 0\) is feasible, so optimality forces \(X_{22} = 0\) and then the constraint forces \(X_{12} = 0\); the dual forces \(y = 0\), \(S = E_{22}\), leaving the ray \(\VEpsC = \{\diag{t,-1} : t \ge 0\}\). At the petal, the slack \(S(y) = E_{22} + \eps^2E_{11} - yA_1 = \bigl[\begin{smallmatrix}\eps^2 & -y\\ -y & 1+y\end{smallmatrix}\bigr]\) is positive semidefinite iff \(y^2 \le \eps^2(1+y)\); since \(b = \eps^3 > 0\), the dual maximizes \(\eps^3y\) at the upper root of that quadratic,
\begin{align*}
    y_\eps = \tfrac12\bigl(\eps^2+\eps\sqrt{\eps^2+4}\bigr) = \eps + \tfrac12\eps^2 + O(\eps^3),
    \qquad
    X_\eps = \frac{u_\eps u_\eps^\top}{\sqrt{\eps^2+4}}, \quad u_\eps := (y_\eps,\ \eps^2)^\top ,
\end{align*}
where \(\det S(y_\eps) = 0\), so \(S(y_\eps)\) has kernel spanned by \(u_\eps\) and complementarity forces \(X = \varsigma\,u_\eps u_\eps^\top\), with \(\barcalA(u_\eps u_\eps^\top) = 2y_\eps\eps^2 - \eps^4 = \eps^3\sqrt{\eps^2+4}\) fixing \(\varsigma = (\eps^2+4)^{-1/2}\). The petal is thus the single point \(\VEpsP = \{Z_\eps\}\), \(Z_\eps := X_\eps - S(y_\eps)\).

For the dislocation, \(\normF{Z_\eps - \diag{t,-1}}^2 = ((Z_\eps)_{11}-t)^2 + 2(Z_\eps)_{12}^2 + ((Z_\eps)_{22}+1)^2\). By \(y_\eps^2 = \eps^2(1+y_\eps)\), \((Z_\eps)_{11} = \eps^2\bigl(\tfrac{1+y_\eps}{\sqrt{\eps^2+4}}-1\bigr) < 0\) for small \(\eps\), so the minimum over \(t \ge 0\) sits at \(t = 0\) and the distance is exact, not just asymptotic; with \((Z_\eps)_{12} = \eps + O(\eps^2)\) and \((Z_\eps)_{22}+1 = -\eps + O(\eps^2)\), it is \(\dist_\calH(Z_\eps,\VEpsC) = \sqrt3\,\eps + O(\eps^2)\).
Finally, \(\thetaEpsP\) and \(\bartheta\) differ only in \(b\) and \(C\), so Theorem~\ref{thm:asr:partial-superposition}\,(\romannumeral1) makes the defect state-free: \(\deltaEpsC - \deltaEpsP \equiv \calA^\dagger(0-\eps^3) - \PAp(-\eps^2E_{11}) = -\tfrac{\eps^3}{3}A_1 + \eps^2E_{11}\), since \(\inprod{A_1}{E_{11}} = 0\). Its two pieces are orthogonal with squared norms \(\eps^6/3\) and \(\eps^4\), so it has size \(\eps^2\sqrt{1+\eps^2/3}\); as \(\deltaEpsP(Z_\eps) = 0\), dividing by the dislocation gives \(\rho_\eps = \eps/\sqrt3\cdot(1+O(\eps)) = \Theta(\eps)\).

\subsubsection{Two afterimages, two exponents}
\label{app:exp:two}
The data, and hence \(\calA^\dagger\), \(\PAp\) and the description of the fixed sets, are those of Example~\ref{exp:sqrt}; only the roles change. The frozen petal is that example's center, so \(\VEpsPOne = V_Z = \{Q_t = \diag{t,-1} : t \ge 0\}\), a ray with endpoint \(Q_0\) and \(\{Q_0\}\) as its only proper face; the center is that example's petal, so \(\VEpsC = \{Z_\eps\}\) with \(Z_\eps = X_\eps - S_\eps\) as computed there, and \(\dist_\calH(Z_\eps, V_Z) = \normF{Z_\eps - Q_0} = \sqrt3\,\eps + O(\eps^2)\), the nearest point being the endpoint since \((Z_\eps)_{11} = -0.5\eps^2 + O(\eps^3) < 0\).

For the second petal \(\thetaEpsPTwo = (\barcalA, 0, C_\eps)\), the primal minimizes \(\inprod{C_\eps}{X} = \eps^2X_{11} + X_{22}\) over \(2X_{12} = X_{22}\), \(X \succeq 0\); both terms are then nonnegative, so \(X = 0\) is the unique optimum, while the dual maximizes \(0 \cdot y\) over \(S(y) = C_\eps - yA_1 \succeq 0\), \ie over the whole interval \(y^2 \le \eps^2(1+y)\), whose endpoints \(y_\pm = 0.5(\eps^2 \pm \eps\sqrt{\eps^2+4}) = \pm\eps + 0.5\eps^2 + O(\eps^3)\) are the roots met in Example~\ref{exp:sqrt}. The images \(X - S(y) = q(y) = -C_\eps + yA_1\) therefore fill a segment of length \(\normF{A_1}(y_+ - y_-) = \sqrt3\,\eps\sqrt{\eps^2+4} = 2\sqrt3\,\eps + O(\eps^3)\). Note \(y_+\) is the center's dual optimum, so \(Z_\eps = q(y_+) + X_\eps\): the segment ends where the collapse point sits, up to the \(O(\eps^2)\) primal part.

Both petals differ from the center in \(b\) and/or \(c\) only, so Theorem~\ref{thm:asr:partial-superposition}\,(\romannumeral1) gives state-free defects \(\deltaEpsC - \deltaEpsPOne \equiv \tfrac{\eps^3}{3}A_1 - \eps^2E_{11}\) and \(\deltaEpsC - \deltaEpsPTwo \equiv \tfrac{\eps^3}{3}A_1\), the second because the two share \(c = C_\eps\); their norms are \(\Delta^1_\eps = \eps^2\sqrt{1+\eps^2/3}\) (orthogonal pieces, as in Example~\ref{exp:sqrt}) and \(\Delta^2_\eps = \eps^3/\sqrt3\). Since \(\deltaEpsPOne\) vanishes on \(V_Z\) and \(\deltaEpsPTwo\) on the segment, \(\norm{\deltaEpsC(z)}_\calH\) is \emph{constant} on each carrier, equal to \(\Delta^1_\eps\), resp.\ \(\Delta^2_\eps\), and~\eqref{eq:asr:sr-def} reduces to dividing by the smallest dislocation.

On \(\SEpsOne\), \(\normF{Q_t - Z_\eps}^2 = (t - (Z_\eps)_{11})^2 + 2(Z_\eps)_{12}^2 + ((Z_\eps)_{22}+1)^2\) increases in \(t \ge 0\), so the infimum sits at \(t = \eps^\alpha\) and equals \(\sqrt{\eps^{2\alpha} + 3\eps^2} + O(\eps^{\alpha+2} + \eps^3)\); with \(\Delta^1_\eps = \eps^2(1+O(\eps^2))\) this gives \(\rho^1_\eps = \Theta(\eps^{2-\alpha})\) for \(\alpha < 1\), \(0.5\eps(1+O(\eps))\) at \(\alpha = 1\), and \(\tfrac{1}{\sqrt3}\eps(1+o(1))\) for \(\alpha > 1\) --- the last two being the saturation, forced by the floor \(\dist_\calH(Z_\eps,V_Z) = \sqrt3\eps(1+O(\eps))\) on the denominator. On \(\SEpsTwo\), \(q(y) - Z_\eps = (y - y_+)A_1 - X_\eps\) gives
\(
    \normF{q(y) - Z_\eps}^2 = 3(y-y_+)^2 - 2(y-y_+)\inprod{A_1}{X_\eps} + \normF{X_\eps}^2,
\)
with \(\inprod{A_1}{X_\eps} = \barcalA X_\eps = \eps^3\) by primal feasibility and \(\normF{X_\eps} = 0.5\eps^2 + O(\eps^3)\). The right side decreases in \(y\) on \([y_-,y_+]\), so the infimum sits at \(y = (1-\eta)y_+\), where \(y - y_+ = -\eta y_+\) and the leading term \(3\eta^2y_+^2 = 3\eta^2\eps^2(1+O(\eps))\) dominates the other two, which are \(O(\eps^4)\); hence \(D^2_\eps = \sqrt3\,\eta\eps + O(\eps^2) > 0\), which also gives \(\SEpsTwo \cap \VEpsC = \emptyset\), and \(\rho^2_\eps = \Delta^2_\eps/D^2_\eps = \tfrac{\eps^2}{3\eta}(1+O(\eps))\). Both regions take \(\vEpsPOne = \vEpsPTwo = 0\) and \(\mu_\eps \equiv 0\) in Definition~\ref{def:asr:asr}, each sitting inside its own petal's fixed-point set. Figure~\ref{fig:exp:two} follows one orbit down the resulting staircase (observed in the run, not implied by the theorems); its \(\eps\)-independent seed sits at the far end \(t = t_1\), an \(O(1)\) distance from \(\VEpsC\), so it outlasts \(N_\eps = \Theta(\eps^{\alpha-2})\), the worst case over \(\SEpsOne\), which is approached only at the near end \(t = \eps^\alpha\).

\subsubsection{Escape to the horizon}
\label{app:exp:escape}
Write \(\barcalA = [-1\ \ 0]\), \(\calA_\eps = [-1\ \ \eps]\), \(c = (0,1)\), \(a = 1/\eps\), and recall the PDHG step of Table~\ref{tab:fom}, \(x^+ = \Pi_\calK[x - \eta_x(c - \calA^*y)]\), \(T_\theta z = (x^+,\, y + \eta_y[b - \calA(2x^+-x)])\). The limit parameter is strongly infeasible: \(\barcalA\R^2_+ = (-\infty,0]\) is closed and \(\dist(1,(-\infty,0]) = 1 > 0\), so \(\bartheta\) admits no KKT point and \(\Fix T_\bartheta = \emptyset\).

Its drift is nevertheless attained, and on an explicit ray. On \(\calR_\bartheta\), where \(x = 0\) and \(y \ge 0\), we have \(c - \barcalA^*y = (y,1)\), so both entries of \(x - \eta_x(y,1)\) are nonpositive, \(x^+ = 0 = x\), and \(y^+ - y = \eta_y\); hence \(\delta_\bartheta \equiv (0,0,\eta_y)\) there --- the primal iterate stays at the origin while the dual moves at constant speed, which reflects the LP dual growing without bound. Seeding at \(z^{(0)} = 0 \in \calR_\bartheta\), the orbit stays on the ray with \(\delta_\bartheta(\zk) \equiv (0,0,\eta_y)\), so Proposition~\ref{prop:asr:forward-drift} forces \(v_\bartheta = (0,0,\eta_y)\). Conversely \(\delta_\bartheta(z) = v_\bartheta\) forces \(x^+ = x\), whose second coordinate \(x_2 = \max\{x_2-\eta_x,0\}\) gives \(x_2 = 0\); the third residual coordinate is then \(\eta_y[1 - \barcalA x] = \eta_y(1+x_1)\), so \(x_1 = 0\), and the first coordinate \(0 = \max\{-\eta_x y,0\}\) gives \(y \ge 0\). The attainment set is therefore exactly \(\calR_\bartheta\), so \(\tagmap(\bartheta) = (+,\fin)\) and \(\VEpsPOne = \{z : \delta_\bartheta(z) = v_\bartheta\} = \calR_\bartheta\): the petal keeps a carrier in the generalized sense even though \(\Fix T_\bartheta\) is empty. Being piecewise affine with polyhedral pieces, these maps have closed \(\ran\delta_\theta\) on an LP, so an \(\inf\) tag cannot occur here.

The center, by contrast, is solvable. Feasibility \(-x_1 + \eps x_2 = 1\), \(x \ge 0\) forces \(x_2 \ge a\), so \(\inprod{c}{x} = x_2\) is minimized uniquely at \(x = (0,a)\), while the dual maximizes \(y\) over \(c - \calA_\eps^*y = (y,\,1-\eps y) \ge 0\), \ie over \([0,a]\), uniquely at \(y = a\). Hence \(\VEpsC = \{(0,a,a)\}\), a singleton, and \(D_\eps = \sqrt2\,a\).

The two moduli follow. For \(z = (0,0,y)\) with \(0 \le y \le a\) the vector \(c - \calA_\eps^*y = (y,\,1-\eps y)\) is again nonnegative, so \(x^+ = 0 = x\) and \(\deltaEpsC(z) = (0,0,\eta_y) = \delta_\bartheta(z) = v_\bartheta\): on \(\SEpsOne\) the center--petal defect vanishes identically, giving \(\mu_\eps \equiv 0\), and the residual there is the \(\eps\)-independent \(\norm{v_\bartheta}_\theta = \sqrt{\eta_y}\). Since \(\norm{z - (0,a,a)}_\calH^2 = a^2 + (y-a)^2\) decreases in \(y\), the infimum over \(\SEpsOne\) sits at \(y = D_\eps^\gamma\), whence \(\rho^1_\eps = \eta_y/\sqrt{a^2+(a-D_\eps^\gamma)^2} = (\eta_y\eps/\sqrt2)(1+O(\eps^{1-\gamma}))\). For the second petal the same KKT computation with \(b = 1-\xi_\eps\) gives \(\VEpsPTwo = \{(0,(1-\xi_\eps)a,\,a)\}\) whenever \(0 < \xi_\eps < 1\); as it differs from the center in \(b\) only, Theorem~\ref{thm:asr:partial-superposition}\,(\romannumeral2) makes the defect state-free and equal to \((0,0,\eta_y\xi_\eps)\), while the dislocation is \(\xi_\eps a\). Both are linear in \(\xi_\eps\), which therefore cancels: \(\rho^2_\eps = \eta_y\eps\) exactly, with \(\xi_\eps\) setting the position of \(\SEpsTwo\) and nothing else.

\subsubsection{SOCP face drop}
\label{app:exp:socp}

\begin{figure}[htbp]
    \centering

    \begin{minipage}{\textwidth}
        \centering
        \begin{minipage}[b]{0.245\textwidth}
            \centering
            \includegraphics[width=\columnwidth]{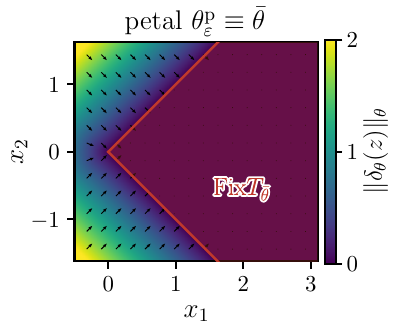}
        \end{minipage}
        \hfill
        \begin{minipage}[b]{0.245\textwidth}
            \centering
            \includegraphics[width=\columnwidth]{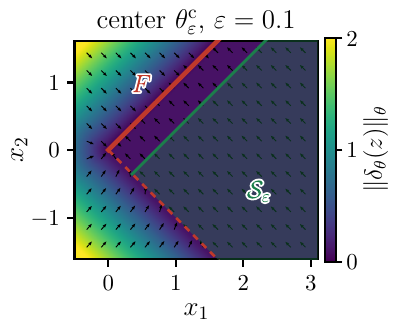}
        \end{minipage}
        \hfill
        \begin{minipage}[b]{0.245\textwidth}
            \centering
            \includegraphics[width=\columnwidth]{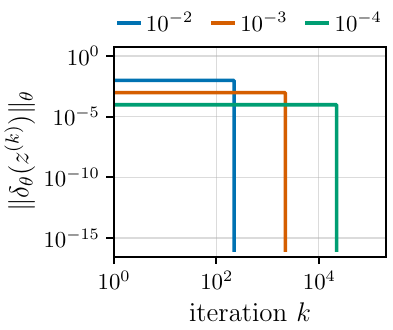}
        \end{minipage}
        \hfill
        \begin{minipage}[b]{0.245\textwidth}
            \centering
            \includegraphics[width=\columnwidth]{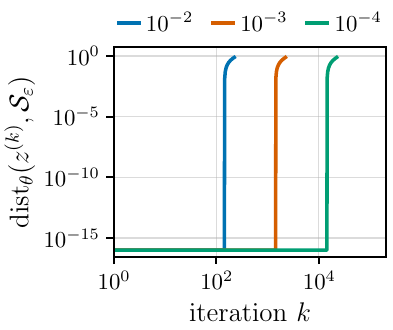}
        \end{minipage}
    \end{minipage}

    \caption{Example~\ref{exp:socp} under PDHG; layout as in Figure~\ref{fig:exp:motivating-pdhg}, on the invariant plane \(\{x_3 = y = 0\}\). The wedge \(\VEpsP\) (solid, then dashed) stays dark although only its edge \(F\) is fixed; from one \(\eps\)-independent seed in \(\calS_\eps\) (shaded) the residual sits at \(\sqrt{2\eta_x}\,\eps\). \label{fig:exp:socp}}
\end{figure}

\begin{example}[SOCP face drop]
    \label{exp:socp}
    Take \(\calK = \calQ^3\), \(\barcalA := [0\ \ 0\ \ 1]\), \(b := 0\), and move only the cost: \(\bartheta := (\barcalA, 0, 0)\), \(\thetaEpsC := (\barcalA, 0, \eps(1,-1,0))\), \(\thetaEpsP \equiv \bartheta\) and \(\normD{\thetaEpsC - \bartheta} = \sqrt2\,\eps\). Both tags equal \((0,\fin)\). The limit cost vanishes, so under PDHG on \(\calH = \R^4\) the petal is the whole wedge \(\VEpsP = \Fix T_\bartheta = \{(x_1,x_2,0,0) : x_1 \ge \abs{x_2}\}\), while the tilt selects one edge, \(\VEpsC = F := \{(r,r,0,0) : r \ge 0\}\), for every \(\eps > 0\) (Appendix~\ref{app:exp:socp}). Theorem~\ref{thm:asr:face-selection} applies, for any \(c_0 > 0\), to any nonempty bounded \(\calS_\eps \equiv \calS \subset \{z \in \VEpsP : \dist_\calH(z,F) \ge c_0\}\), on which the defect is exactly \(-\eta_x\eps(1,-1,0,0)\), giving \(\rho_\eps {}\le \sqrt2\eta_x\eps/c_0 = \Theta(\normD{\thetaEpsC-\bartheta})\); see Figure~\ref{fig:exp:socp}.
\end{example}
At \(\eps = 0\) the objective vanishes, so the optimal set is the \emph{entire} feasible slice \(\calQ^3 \cap \{x_3 = 0\} = \{x_1 \ge \abs{x_2}\}\); moving only \(c\) leaves that slice fixed, so the tilt merely selects one of its two extreme rays, while the dual optimum \(y = 0\) stays unique. Hence \(\VEpsC = F \subsetneq \VEpsP\) exactly for every \(\eps > 0\), with no limit taken. On \(\calS\) the constant is exact: \(\dist_\calH((x_1,x_2,0,0),F) = \abs{x_1-x_2}/\sqrt2 \ge c_0\) forces \(x_1 - x_2 \ge \sqrt2c_0\), so for small \(\eps\) the point \(x - \eta_x\eps(1,-1,0)\) still lies in \(\calQ^3\), the projection is inactive, and \(\deltaEpsC \equiv -\eta_x\eps(1,-1,0,0)\) has norm exactly \(\sqrt2\eta_x\eps\); dividing by \(c_0\) gives \(\rho_\eps {}\le \sqrt2\eta_x\eps/c_0\).

\subsubsection{SOCP zero gap, not attained}
\label{app:exp:nonatt}

\begin{figure}[htbp]
    \centering

    \begin{minipage}{\textwidth}
        \centering
        \begin{minipage}[b]{0.245\textwidth}
            \centering
            \includegraphics[width=\columnwidth]{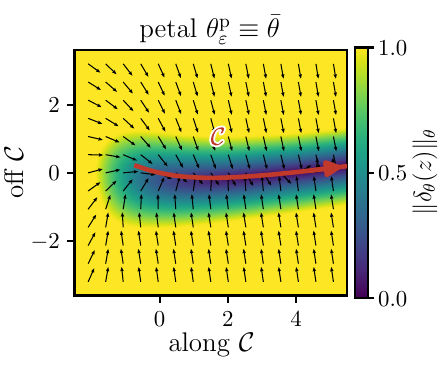}
        \end{minipage}
        \hfill
        \begin{minipage}[b]{0.245\textwidth}
            \centering
            \includegraphics[width=\columnwidth]{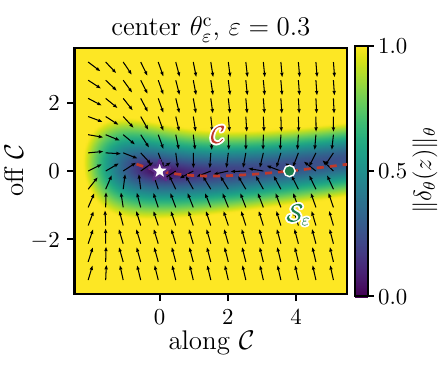}
        \end{minipage}
        \hfill
        \begin{minipage}[b]{0.245\textwidth}
            \centering
            \includegraphics[width=\columnwidth]{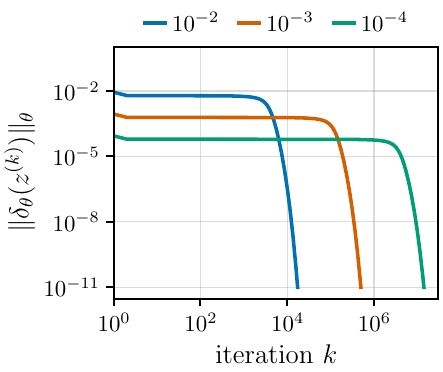}
        \end{minipage}
        \hfill
        \begin{minipage}[b]{0.245\textwidth}
            \centering
            \includegraphics[width=\columnwidth]{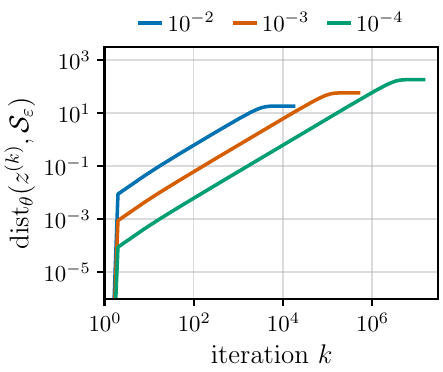}
        \end{minipage}
    \end{minipage}

    \caption{Example~\ref{exp:nonatt} under ADMM; layout as in Figure~\ref{fig:exp:motivating-pdhg}, on the plane through \(\VEpsC\) spanned by the chord to the seed and \(\calC\)'s transverse direction, so both marked points are exactly in-slice. \emph{Left two:} the field has \emph{no} zero --- \(\calC\) merely darkens and runs off the frame (arrow); then a zero appears at \(\VEpsC\) (star), with the seed \(\calS_\eps\) (dot) further out at \(\lambda = 2\). \emph{Right two:} one orbit per \(\eps = 10^{-2},10^{-3},10^{-4}\). \label{fig:exp:nonatt}}
\end{figure}

\begin{example}[SOCP zero gap, not attained]
    \label{exp:nonatt}
    Let \(\calK = \calQ^3\), \(\barcalA := [0\ \ 0\ \ 1]\), \(b := 1\), \(u := (0,1,0)\), and move the cost only, freezing the petal at the limit: \(\bartheta = \thetaEpsP :\equiv (\barcalA,\, 1,\, \bar c)\) with \(\bar c := (1,-1,0)\), \(\thetaEpsC := (\barcalA,\, 1,\, \bar c + \eps u)\) and \(\normD{\thetaEpsC - \bartheta} = \eps\).
    The limit minimizes \(x_1 - x_2\) over \(x_3 = 1\), \(x \in \calQ^3\): the value is \(0\), attained by the dual at \(y = 0\) but by no feasible \(x\), since \(x_1 - x_2 \ge \sqrt{x_2^2+1} - x_2 > 0\). So \(\bartheta \notin \calD_\KKT\) and \(\Fix T_\bartheta = \emptyset\) while \(v_\bartheta = 0\): \(\tagmap(\bartheta) = (0,\inftag)\), a tag the maps of Table~\ref{tab:fom} cannot exhibit on an LP. Writing \(r_t := \sqrt{t^2+1}\), the images \(Z_t\) of \(x_t = (r_t,t,1)\), \(s_t = (1,-t/r_t,-1/r_t)\) trace a curve \(\calC\) carrying \(\norm{\delta_\bartheta(Z_t)}_\calH = 1 - t/r_t = \Theta(t^{-2})\), which tends to \(0\) without reaching it. Since \(\VEpsP = \emptyset\), \(\mu_\eps > 0\) is forced and \(\calS_\eps\) is carried by a residual sublevel set. The petal's residual \(\Theta(t^{-2})\) and the state-free defect \(\deltaEpsC - \deltaEpsP \equiv -\eps u\) match at \(t \asymp \eps^{-1/2}\), which is where the center's unique fixed point sits: \(\VEpsC = \{Z_{t^\star_\eps}\}\) lies \emph{on} \(\calC\) with \(t^\star_\eps = 2^{-1/2}\eps^{-1/2}(1+O(\eps))\), escaping as \(D_\eps = \eps^{-1/2}(1+O(\eps))\). Seeding there, at \(\calS_\eps := \{Z_{t_\eps}\}\) with \(t_\eps := \lambda\eps^{-1/2}\) and any fixed \(\lambda > 0\), gives \(\mu_\eps = \Theta(\eps)\) and \(\rho_\eps = \Theta(\eps^{3/2})\), so \(N_\eps = \Theta(\eps^{-3/2})\); see Figure~\ref{fig:exp:nonatt}. Theorem~\ref{thm:asr:escape-to-horizon} applies, though its own carrier sits inside \(\bbB_\calH(0,D_\eps^\gamma)\) whereas \(\calS_\eps\) sits \emph{at} the horizon; confined to the ball the same construction yields only \(\Theta(\eps^{\gamma+1/2})\).
\end{example}
Here \(\barcalA\barcalA^* = 1\), \(\barcalA^\dagger b = (0,0,1)\), \(\PA = \diag{0,0,1}\), \(\PAp = \diag{1,1,0}\), and \(\sigma = \tau = 1\) makes the metric Euclidean. Dual feasibility \(\bar c - \barcalA^*y = (1,-1,-y) \in \calQ^3\) reads \(1 \ge \sqrt{1+y^2}\) and forces \(y = 0\); with the primal value unattained, \(\KKT(\bartheta) = \emptyset\) and Assumption~\ref{ass:asr:T}\,(\romannumeral1) gives \(\Fix T_\bartheta = \emptyset\).

The triples are almost-KKT: \(x_t, s_t \in \partial\calQ^3\), \(\inprod{x_t}{s_t} = 0\) and \(\barcalA x_t = 1\), so \(\Pi_\calK(Z_t) = x_t\) and only the dual residual survives, \(\delta_\bartheta(Z_t) = (1 - t/r_t)u\). Hence \(\inf_z \norm{\delta_\bartheta(z)}_\calH = 0 = v_\bartheta\), while \(\delta_\bartheta(z) = 0\) would produce a KKT triple: \(v_\bartheta \notin \ran\delta_\bartheta\) and \(\tagmap(\bartheta) = (0,\inftag)\). None of the Table~\ref{tab:fom} maps can do this on an LP: they are piecewise affine with polyhedral pieces, so \(\ran\delta_\theta\) is closed and the minimal displacement is attained.

The center is solvable: minimizing \(x_1 - (1-\eps)x_2\) forces \(x_2/r_{x_2} = 1-\eps\), giving the unique optima \(x^\star_\eps = (\kappa_\eps^{-1},\, (1-\eps)\kappa_\eps^{-1},\, 1)\) with \(\kappa_\eps := \sqrt{2\eps-\eps^2}\) and \(y^\star_\eps = \kappa_\eps\). Both lie on \(\calC\), so \(\VEpsC = \{Z_{t^\star_\eps}\}\) with \(t^\star_\eps = (1-\eps)/\kappa_\eps\) and \(D_\eps = \sqrt2\,t^\star_\eps(1+o(1))\).

For the moduli, \(t_\eps = \lambda\eps^{-1/2}\) gives \(\mu_\eps = 1 - \lambda/\sqrt{\lambda^2+\eps} = \eps/(2\lambda^2) + O(\eps^2)\), which is Definition~\ref{def:asr:asr} with \(\vEpsP = 0\); since \(\deltaEpsC = \deltaEpsP - \eps u\) on \(\calC\), the center residual is \(\abs{\mu_\eps - \eps}\), a \emph{signed} cancellation that is exact to leading order at \(\lambda = 2^{-1/2}\). Along \(\calC\), \(\tfrac{d}{dt}Z_t = (1,1,0) + O(t^{-2})\), so \(\norm{Z_t - Z_{t'}}_\calH = \sqrt2\abs{t-t'}(1+O(\min(t,t')^{-2}))\) and the dislocation is \(\sqrt2\abs{\lambda - \lambda^\star_\eps}\eps^{-1/2}(1+O(\eps))\) with \(\lambda^\star_\eps := t^\star_\eps\eps^{1/2} = 2^{-1/2}(1 - \tfrac34\eps + O(\eps^2))\). Dividing cancels \(\abs{\lambda - 2^{-1/2}}\) and leaves \(\rho_\eps = (\lambda + 2^{-1/2})(\sqrt2\lambda^2)^{-1}\eps^{3/2}(1+o(1))\); at \(\lambda = 2^{-1/2}\) both leading terms vanish and one further order returns the same formula, so it is continuous there. A \emph{fixed} \(\lambda\) is admissible for all small \(\eps\), whereas an interval of \(\lambda\)'s straddling \(2^{-1/2}\) would eventually contain the moving \(\lambda^\star_\eps\) and meet \(\VEpsC\).

\subsubsection{SDP with a non-closed image}
\label{app:exp:adrift}

\begin{figure}[htbp]
    \centering

    \begin{minipage}{\textwidth}
        \centering
        \begin{minipage}[b]{0.245\textwidth}
            \centering
            \includegraphics[width=\columnwidth]{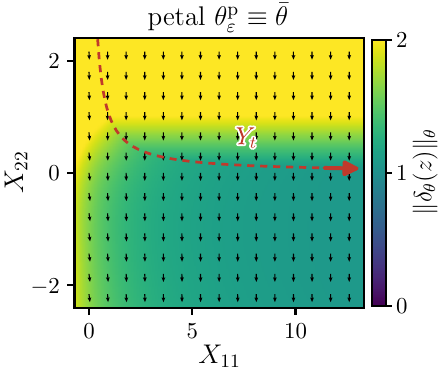}
        \end{minipage}
        \hfill
        \begin{minipage}[b]{0.245\textwidth}
            \centering
            \includegraphics[width=\columnwidth]{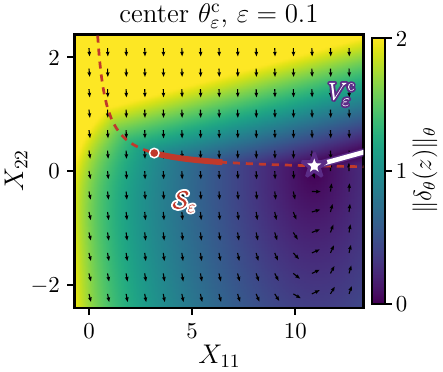}
        \end{minipage}
        \hfill
        \begin{minipage}[b]{0.245\textwidth}
            \centering
            \includegraphics[width=\columnwidth]{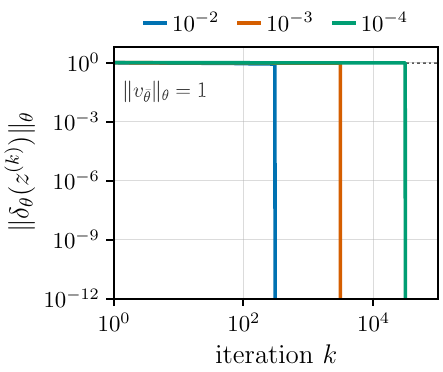}
        \end{minipage}
        \hfill
        \begin{minipage}[b]{0.245\textwidth}
            \centering
            \includegraphics[width=\columnwidth]{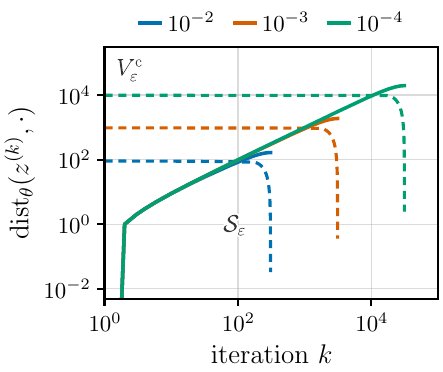}
        \end{minipage}
    \end{minipage}

    \caption{Example~\ref{exp:adrift} under ADMM; layout as in Figure~\ref{fig:exp:motivating-pdhg}, on the plane \(\{X_{12} = 1\}\), which contains the boundary curve \(Y_t\) and the center's optimal curve exactly; the horizontal axis is compressed, the tip escaping as \(\eps^{-1}\) while the transverse structure stays \(O(1)\). \emph{Left two:} the color never darkens below \(\norm{v_\bartheta}_\calH = 1\) and \(Y_t\) (dashed) runs off the frame (arrow); then the zero set appears as the unbounded curve \(\VEpsC\) (white), its tip (star) on \(Y_t\), with \(\calS_\eps\) (solid red, seed dot at its near end) the arc at \(t \asymp \eps^{-1/2}\). \emph{Right two:} one orbit per \(\eps\), seeded at the arc's near endpoint; solid curves give \(\dist_\theta(\zk,\calS_\eps)\) and dashed ones \(\dist_\theta(\zk,\VEpsC)\). \label{fig:exp:adrift}}
\end{figure}

\begin{example}[SDP with a non-closed image]
    \label{exp:adrift}
    Let \(\calX = \sym{2}\), \(\calK = \psd{2}\), \(\calY = \R^2\), \(b := (1,-1)\), and move the operator only, freezing the petal at the limit: \(\calA_\eps X := (X_{12},\ X_{22} - \eps X_{11})\), \(\bartheta = \thetaEpsP :\equiv (\calA_0,\, b,\, 0)\) and \(\thetaEpsC := (\calA_\eps,\, b,\, 0)\), so \(\normD{\thetaEpsC - \bartheta} = \eps\).
    The cost vanishes, so the optimal set is the feasible set. The operator stays nondegenerate, \(\calA_\eps\calA_\eps^* = \diag{0.5,\ 1+\eps^2} \succeq_\calY 0.5\Id\) for every \(\eps \ge 0\); what fails is its interaction with the cone boundary, \(\calA_0(\psd{2}) = \{(u,v) : v > 0\} \cup \{(0,0)\}\), which is not closed. Since \(b\) has \(v\)-coordinate \(-1\), the nearest image points are \((1,v)\) with \(v \downarrow 0\), so \(\dist(b, \calA_0(\psd{2})) = 1\) is realized by no \(X \succeq 0\): \(\Fix T_\bartheta = \emptyset\), \(v_\bartheta = -E_{22} \notin \ran\delta_\bartheta\) with \(\norm{v_\bartheta}_\calH = 1\), and \(\tagmap(\bartheta) = (+,\inftag)\), the last cell. The drift is approached along the rank-one \emph{boundary curve} \(Y_t := \bigl[\begin{smallmatrix}t&1\\1&1/t\end{smallmatrix}\bigr] \in \partial\psd{2}\), where \(\delta_\bartheta(Y_t) = -(1+\tfrac1t)E_{22}\) \emph{exactly}. The center is solvable, \(\tagmap(\thetaEpsC) = (0,\fin)\), with optimal set the unbounded ray \(\VEpsC = \{X_p := \bigl[\begin{smallmatrix}p&1\\1&\eps p-1\end{smallmatrix}\bigr] : p \ge r_\eps\}\), \(r_\eps := 0.5(1+\sqrt{1+4\eps})/\eps = \eps^{-1}+1+O(\eps)\), whose tip \(X_{r_\eps} = Y_{r_\eps}\) sits on the boundary curve and escapes, \(D_\eps = r_\eps + r_\eps^{-1}\).

    On the boundary curve every residual is the scaled primal infeasibility, \(\norm{\deltaEpsC(Y_t)}_\calH = \abs{g_\eps(t)}/\sqrt{1+\eps^2}\) with \(g_\eps(t) := 1 + 1/t - \eps t\): the limit's own \(1/t\) and the perturbation's \(\eps t\) are of the same size when \(t \asymp \eps^{-1/2}\), and cancel at \(t = \eps^{-1/2}\), where \(g_\eps = 1\) exactly. Fixing \(0 < c_0 < c_1\), the \emph{diverging} arc \(\calS_\eps := \{Y_t : c_0\eps^{-1/2} \le t \le c_1\eps^{-1/2}\}\) has \(\mu_\eps = \eps^{1/2}/c_0 = \Theta(\eps^{1/2})\) and \(\rho_\eps = \eps(1 + \mu_\eps + O(\eps)) = \Theta(\eps)\), hence \(N_\eps = \Theta(\eps^{-1})\), the order Theorem~\ref{thm:asr:escape-to-horizon} displays. The carrier itself \emph{diverges}, at rate \(\eps^{-1/2}\) while staying far short of the tip at \(\eps^{-1}\), and is the first region here that is positive-dimensional \emph{and} escapes entirely to infinity. Its center residual converges to \(\norm{v_\bartheta}_\calH = 1\) uniformly on the arc, \(\sup_{z \in \calS_\eps}\abs{\norm{\deltaEpsC(z)}_\calH - 1} = O(\eps^{1/2})\), while \(\mu_\eps = \sup_{z \in \calS_\eps}\norm{\delta_\bartheta(z) - v_\bartheta}_\calH = \eps^{1/2}/c_0\) records the petal's own gap; and the perturbation moves \(\calA\), so no \(U\) with \(\calA\) fixed contains the families; Theorem~\ref{thm:asr:partial-superposition} is therefore unavailable and the defect is state dependent. See Figure~\ref{fig:exp:adrift}.
\end{example}
Throughout \(\sigma = \tau = 1\), so \(\norm{\cdot}_\theta = \normF{\cdot}\) and \(\calH = \calX = \sym{2}\); with \(c = 0\) and \(\Pi_+ := \Pi_{\psd{2}}\), \(\Pi_- := \Id - \Pi_+\), Table~\ref{tab:fom} reads \(\delta_\theta(Z) = -\PA(\Pi_+(Z) - \calA^\dagger b) - \PAp(\Pi_-(Z))\). Since full row rank is an open condition, \(\calA_0\calA_0^* = \diag{0.5,1} \succ_\calY 0\) lets us fix a compact \(U \ni \bartheta\) on which \(\calA\calA^* \succeq_\calY \tfrac14\Id\), and Table~\ref{tab:fom} then gives Assumption~\ref{ass:asr:T} throughout \(U\), the limit included.

\emph{The image and the tag.} For \(X \succeq 0\) we have \(X_{22} \ge 0\), and \(X_{22} = 0\) forces \(X_{12} = 0\), while every \(v > 0\) is hit by \(X = \bigl[\begin{smallmatrix}u^2/v & u\\ u & v\end{smallmatrix}\bigr]\); that is the stated image, whose closure contains \((1,0)\) although the set does not. On the boundary curve, \(\det Y_t = 0\) and \(t > 0\) give \(Y_t \succeq 0\), so \(\Pi_+(Y_t) = Y_t\) and the map collapses to the scaled primal infeasibility,
\begin{align}
    \label{eq:app:exp:onbdry}
    \delta_\theta(Y_t) = -\PA(Y_t - \calA^\dagger b) = -\calA^\dagger(\calA Y_t - b),
\end{align}
which at \(\bartheta\) is \(-(1+\tfrac1t)E_{22} \to -E_{22}\). Conversely \(\normF{\calA_0^\dagger(u,v)}^2 = 2u^2+v^2 \ge \norm{(u,v)}_2^2\), the pseudo-inverse being expansive here, and the two terms of the map lie in the orthogonal subspaces \(\ran\calA_0^*\) and \(\ker\calA_0\); dropping the second and writing \(X := \Pi_+(Z)\) gives, for \emph{every} \(Z \in \calH\),
\begin{align}
    \label{eq:app:exp:lb}
    \normF{\delta_\bartheta(Z)} \ \ge\ \normF{\calA_0^\dagger(\calA_0X - b)} \ \ge\ \norm{\calA_0X - b}_2 \ \ge\ \dist(b,\calA_0(\psd{2})) = 1 .
\end{align}
So the infimum is exactly \(1\). Now \(\cl{\ran\delta_\bartheta}\) is convex by Proposition~\ref{prop:asr:forward-drift}, so its minimum-norm element is unique and equals \(-E_{22}\); equality in~\eqref{eq:app:exp:lb} would need an image point at distance exactly \(1\), so \(v_\bartheta \notin \ran\delta_\bartheta\) and \(\tagmap(\bartheta) = (+,\inftag)\).

\emph{The center.} Solving \(\calA_\eps X = b\) forces \(X_{12} = 1\) and \(X_{22} = \eps X_{11} - 1\), so with \(p := X_{11}\) the feasible points are exactly the \(X_p\), and \(X_p \succeq 0\) iff \(p \ge r_\eps\). Taking \(y = 0\), \(S = 0\) makes each of them a KKT point, and it is the only dual slack: \(S = -\calA_\eps^*y \succeq 0\) needs \(\eps y_2 \ge 0\) and \(-y_2 \ge 0\) at once, forcing \(y_2 = 0\), after which \(\det S = -y_1^2/4 \ge 0\) forces \(y_1 = 0\). Hence \(\VEpsC\) is as stated and \(\tagmap(\thetaEpsC) = (0,\fin)\). At \(p = r_\eps\) the determinant vanishes, so \(X_{r_\eps} = Y_{r_\eps}\), and \(\normF{X_p}\) increases on \(p \ge r_\eps\), making the tip the nearest fixed point to the origin, with \(D_\eps = r_\eps + r_\eps^{-1}\).

\emph{The two moduli.} By~\eqref{eq:app:exp:onbdry} and \(v_\bartheta = -E_{22}\), \(\norm{\delta_\bartheta(Y_t) - v_\bartheta}_\calH = 1/t\), largest on \(\calS_\eps\) at the near endpoint: \(\mu_\eps = \eps^{1/2}/c_0\), and \(\VEpsP = \emptyset\) leaves no smaller choice available. For the center, \(\calA_\eps Y_t - b = (0,\,g_\eps(t))\) and \(\calA_\eps^\dagger(0,w)\) has norm \(\abs{w}/\sqrt{1+\eps^2}\), which is the stated residual; writing \(t = c\eps^{-1/2}\) with \(c \in [c_0,c_1]\) gives \(g_\eps = 1 + \eps^{1/2}(c^{-1}-c)\), within \(O(\eps^{1/2})\) of \(1\) \emph{uniformly} in \(c\), which is the uniform rise. The nearest fixed point to \(Y_t\) is again the tip, since \(r_\eps \sim \eps^{-1}\) dwarfs \(c_1\eps^{-1/2}\); so \(\calS_\eps \cap \VEpsC = \emptyset\) and \(r_\eps - t = \eps^{-1}(1 - c\eps^{1/2} + O(\eps))\). The \(-c\eps^{1/2}\) of the denominator cancels that of the numerator, leaving \(\norm{\deltaEpsC(Y_t)}_\calH/\dist_\calH(Y_t,\VEpsC) = \eps(1 + \eps^{1/2}/c + O(\eps))\) uniformly in \(c\); the supremum is therefore pinned from both sides, with no monotonicity needed, at \(\rho_\eps = \eps(1 + \mu_\eps + O(\eps))\), the surviving correction contributed by the near end. Theorem~\ref{thm:asr:escape-to-horizon} applies, its conclusion \(O(D_\eps^{-1} + D_\eps^{\gamma-1}\normD{\thetaEpsC-\bartheta}) = O(\eps)\) being the order realized here, and the arc meets the theorem's near-minimizer test once \(\nu_\eps := \mu_\eps\) is chosen, since \(\sup_{\calS_\eps}\normF{Z} \asymp \eps^{-1/2}\) sits inside \(\bbB_\calH(0,D_\eps^\gamma)\) for every fixed \(\gamma > 0.5\). We claim only that the arc \emph{matches} the displayed order.

\end{document}